\documentclass[11pt]{amsart}
\usepackage[hmargin = 0.75 in, vmargin = 0.75 in]{geometry}
\usepackage{amssymb, amsmath,amsthm}
\newtheorem{thm}{Theorem}[section]

\newtheorem{lem}[thm]{Lemma}
\theoremstyle{definition}
\newtheorem{rem}[thm]{Remark}

\newcommand{\ra}{\rightarrow}
\newcommand{\bk}{\backslash}
\newcommand{\mc}{\mathcal}
\newcommand{\mb}{\mathbb}
\newcommand{\sg}{\sigma}
\newcommand{\R}{\mb{R}}

\newcommand{\N}{\mb{N}}

\newcommand{\llf}{\left\lfloor}
\newcommand{\e}{\epsilon}
\newcommand{\rrf}{\right\rfloor}
\newcommand{\mbf}{\boldsymbol}

\begin{document}
\title[Restricted Prime Factors I]{On the Distribution of Integers with Restricted Prime Factors I}
\normalsize
\author[A. P. Mangerel]{Alexander P. Mangerel}
\address{Department of Mathematics\\ University of Toronto\\
Toronto, Ontario, Canada}
\email{sacha.mangerel@mail.utoronto.ca}
\begin{abstract}
Let $E_0,\ldots,E_n$ be a partition of the set of prime numbers, and define $E_j(x) := \sum_{p \in E_j \atop p \leq x} \frac{1}{p}$. Define $\pi(x;\mbf{E},\mbf{k})$ to be the number of integers $n \leq x$ with $k_j$ prime factors in $E_j$ for each $j$. Basic probabilistic heuristics suggest that $x^{-1}\pi(x;\mbf{E},\mbf{k})$, modelled as the distribution function of a random variable, should satisfy a joint Poisson law with parameter vector $(E_0(x),\ldots,E_n(x))$, as $x \ra \infty$. We prove an asymptotic formula for $\pi(x;\mbf{E},\mbf{k})$ which contradicts these heuristics in the case that for each $j$, $E_j(x)^2 \leq k_j \leq \log^{\frac{2}{3}-\e} x$ for each $j$ under mild hypotheses. As a particular application, we prove an asymptotic formula regarding integers with prime factors from specific arithmetic progressions, which generalizes a result due to Delange.
\end{abstract}
\maketitle
\section{Introduction}
The prime number theorem, which provides asymptotic information on the distribution of all primes up to a specific bound, is a central result in number theory. A natural generalization of the notion of a prime number is that of an almost prime number, in which the number of its prime factors is constrained, and the question of determining the distribution of almost primes up to a specific bound is also natural.  Let $\pi_k(x) := |\left\{n \leq x : \omega(n) = k\right\}|$, where $\omega(n)$ is the number of distinct prime factors of an integer $n$, and $k \in \mb{N}$. As a corollary of the prime number theorem, Landau \cite{HaW} proved that, asymptotically, $ \pi_k(x) = (1+o(1))\frac{x}{\log x} \frac{(\log_2 x)^{k-1}}{(k-1)!}$, where $\log_2 x := \log(\log x)$ for $x > 1$, though the error term in this formula is only effective uniformly for small values of the parameter $k$. A number of papers appeared subsequently in which upper and lower estimates for these quantities, uniform over larger intervals, were established, notably by Hardy and Ramanujan \cite{HR}, who showed that there exist fixed constants $C_1, C_2 > 0$ such that for every $k$ we have 
\begin{equation} \label{HarRam}
\pi_k(x) \leq C_1 \frac{x}{\log x} \frac{(\log_2 x + C_2)^{k-1}}{(k-1)!},
\end{equation}
the latter result established by elementary means. Subsequently, Sathe was able to find a uniform asymptotic, akin to Landau's result, over a substantially larger range for the parameter $k$, which was improved and simplified shortly thereafter using analytic methods by Selberg. He showed that when $y := \frac{k}{\log_2 x} \ll 1$,
\begin{equation} \label{SelSat}
\pi_k(x) = \frac{1}{\Gamma(y+1)}\prod_{p}\left(1+\frac{y}{p-1}\right)\left(1-\frac{1}{p}\right)^y\frac{x}{\log x}\frac{\log_2^{k-1} x}{(k-1)!}\left(1+O\left(\frac{1}{\log_2 x}\right)\right);
\end{equation}
in other words, the range of uniformity of these results in $k$ is $1 \leq k \leq C \log_2 x$, where $C > 0$ is any constant.  It is well-known that the number of prime factors of an integer $n \leq x$ cannot exceed $(1+o(1))\frac{\log x}{\log_2 x}$ (see I.5 in \cite{Ten2}); hence, it is of interest to determine an asymptotic formula uniform for all $1 \leq k \ll \frac{\log x}{\log_2 x}$.  The intent to increase the above range of uniformity for the parameter $k$ was initiated by Hensley \cite{He} and Pomerance \cite{Pom}, and the widest range of parameters determined to date (and the sharpest results as well) are essentially due to Hildebrand and Tenenbaum \cite{HiT}, who utilized an analytic argument based on the saddle-point method (which we describe in what follows) to establish, essentially, the existence of constants $\rho$ and $\sg$ such that for $1 \leq k \ll \frac{\log x}{\log_2^2x}$,
\begin{equation} \label{HilTen}
\pi_k(x) = \frac{\rho^{-k}x^{\sg}F(\rho;\sg)}{(\log x)kw(k)w(\rho)}\left(1+O\left(\frac{1}{\log(\log x/(k\log(k+1)))}\right)\right),
\end{equation}
where, given $z,s \in \mb{C}$, $F(z;s) := \prod_{p} \left(1+\frac{z}{p^s-1}\right)$ and $w(t) := \Gamma(t)t^{-t}e^t \sim (2\pi t)^{\frac{1}{2}}$, by Stirling's approximation. Using mostly the same argument, Kerner was able to extend their result to the entire range $1\leq k \ll (1-\e)\frac{\log x}{\log_2 x}$ for any fixed $\e>0$ in his thesis \cite{Ker} .  \\
The above results depend crucially on the existence of the asymptotically smooth distribution function for the primes given by the prime number theorem. For instance, to prove \eqref{SelSat}, Selberg employed the function $F(z;s)\zeta(s)^{-z}$, which is holomorphic in any zero-free region to the left of the vertical line $s = 1$, except for an (essential) singularity at $s = 1$. In particular, the fact that $G(z;s) := F(z;s)(s-1)^z$ is holomorphic at $s = 1$ is crucial (see \cite{Sel}), and it is due to the prime number theorem that there is a direct connection between $\zeta(s)$ having residue 1 at $s = 1$ and, for instance, the asymptotics of sums such as $\sum_{p \leq x} \frac{\log p}{p}$. It will be evident from what follows that specifically the estimate $\sum_{p \leq x} \frac{1}{p} = \log_2 x + O(1)$ is of essence in these problems (and the smooth function $\log_2 x$ is a consequence of $\pi(x)$ being asymptotically smooth). \\
It is natural, then, to consider whether these results have analogues in cases of the following type. Let $E \subset \mc{P}$ and let $\omega_E(n)$ denote the number of distinct prime factors of $n$ that belong to the set $E$. Of course, when the function $\pi_E(x) := |\left\{p \leq x : p \in E\right\}|$ is similarly smooth like $\pi(x)$ (for instance, when the set $E$ has Dirichlet density, in which case $\sum_{p \leq x \atop p \in E} \frac{1}{p} = (1+o(1))\lambda \sum_{p\leq x} \frac{1}{p}$, where $\lambda \geq 0$ is a constant), these problems are more tractable.  Delange considered some such applications of (a generalized form of) Selberg's method in \cite{D}.  On the other hand, if we do not assume any hypotheses regarding the regularity of distribution of $E$, the arguments of the papers thus far mentioned are not applicable. \\
The most important results in this direction to date are due to Hal\'{a}sz \cite{Ha} (strictly, speaking, Hal\'{a}sz' theorem follows from an analysis of \emph{completely} additive functions, a set to which $\omega_E(n)$ does not belong; simple modifications to his arguments were made in a paper by Norton \cite{N} in order to apply it to additive functions in general). Let $E(x) := \sum_{p \in E \atop p\leq x} \frac{1}{p}$ and set $\pi(x;E;k) := |\left\{n \leq x : \omega_E(n) = k\right\}|$. Hal\'{a}sz proved that when $\delta > 0$ is fixed and $\delta E(x) \leq k \leq (2-\delta)E(x)$, and provided that $E(x) \ra \infty$ as $x \ra \infty$, 
\begin{equation} \label{HalEst}
\pi(x;E,k) \asymp_{\delta} x \frac{E(x)^k}{k!}e^{-E(x)};
\end{equation}
moreover, in the case that $k = (1+o(1))E(x)$, he proved, in fact, the asymptotic formula 
\begin{equation} \label{HalAsymp}
\pi(x;E,k) = x\frac{E(x)^k}{k!}e^{-E(x)}\left(1+O_{\delta}\left(\frac{|k-E(x)|}{E(x)} + \frac{1}{E(x)^{\frac{1}{2}}}\right)\right).
\end{equation}
In analogy to \eqref{SelSat}, and a fortiori given the lack of asymptotic estimates in general for $1 \leq k \leq CE(x)$ for $C \geq 2$, it is of interest to consider asymptotic formulae for $\pi(x;E,k)$ that are effective in a larger range for the parameter $k$. \\
A further problem presents itself once one restricts to an investigation of the prime factors of an integer from a \emph{proper} subset of the primes: we may consider the problem of determining how many integers possess a fixed number of prime factors in \emph{multiple} different sets, again hopefully with minimal restrictions on the subsets.  To this end, if $n \in \N$, $E_1,\ldots,E_n \subset \mc{P}$ are disjoint and $k_1,\ldots,k_n \geq 0$, we may define
\begin{equation*}
\pi(x;\mbf{E},\mbf{k}) := |\left\{m\leq x: \omega_{E_j}(m) = k_j \ \forall \ 1 \leq j \leq n\right\}|.
\end{equation*}
Delange was the first to consider such functions, establishing results for $\pi(x;\mbf{E},\mbf{k})$ when each set $E_j$ has Dirichlet density \cite{D}; however, as he implements Selberg's method, his results are only valid for $k_j \ll E_j(x)$ for $j = 1,2$.  Otherwise, to the author's knowledge, the only other investigation of this type is due to Tudesq \cite{Tu} who showed that the formula 
\begin{equation}
\pi(x;\mbf{E},\mbf{k}) \ll xe^{-\sum_{1 \leq j \leq n} E_j(x)}\prod_{1 \leq j \leq n} \frac{(E_j(x)+\mu)^{k_j}}{k_j!}, \label{UppEs}
\end{equation}
the analogue of \eqref{HarRam} in this context, holds uniformly for all tuples $(k_1,\ldots,k_n)$ and $(E_1,\ldots,E_n) \subset \mc{P}^n$, for $\mu > 0$ some constant. Furthermore, no general, lower bounds analogous to \eqref{UppEs} exist in any case.  The main objective of the present paper is to find precise results in this vein. \\
The probabilistic lens, when applied in number theoretic problems, has proven fruitful in producing arithmetic models that provide useful heuristics. Of particular note is the analogy of $\omega_E(n)$ to a sum of independent random variables $\left\{X_p\right\}_{p \in E}$ defined on the set $\mb{N} \cap [1,x]$, taking values 1 or 0 according to whether or not $n$ is divisible by a prime $p \in E$ (with probability $\frac{1}{p}$ that $p|n$).  In some sense, $x^{-1}\pi(x;E,k)$ encapsulates the probability (depending on $x$) that for a given integer $n \leq x$, $n$ belongs to the intersection of $k$ events within which, for each $p|n$ and $p \in E$, the random variables $X_p = 1$.  Such probabilities are described by a Poisson process, wherein there exists a parameter $\lambda > 0$ such that the probability that a non-negative integer-valued random variable take the value $k$ is $\frac{\lambda^k}{k!}e^{-\lambda}$ for each $k \in \mb{N}_0$. The form of Hal\'{a}sz' results indicates that for $x$ sufficiently large, $g_x(k) := x^{-1}\pi(x;E,k)$ is a probability mass function for a Poisson process with parameter $E(x)$. We note, moreover, that one expects that any two disjoint sets of primes will have, associated to them, independent Poisson processes, and so naturally, the probabilities that an integer $m$ have $k_j$ prime factors in $E_j$ should also be independent. This provides the heuristic that
\begin{equation}
\pi(x;\mbf{E},\mbf{k}) \approx x\prod_{1 \leq j \leq n} \frac{E_j(x)^{k_j}}{k_j!}e^{-E_j(x)}, \label{HEUR}
\end{equation}
which is in line with the upper bound \eqref{UppEs}.  On the other hand, the shortcoming of this reasoning stems from the fact that divisibility by primes is not entirely independent, as is seen, for example, by noting that for an integer $n$ being divisible by two distinct primes $p,q > \sqrt{n}$ has probability zero.  Some correction factor in \eqref{HEUR} may be necessary in order to account for effects of this latter type. \\
We will prove precise results that contradict \eqref{HEUR} when $k_j \gg E_j(x)^2$ in a variety of cases in which the collection $\{E_1,\ldots, E_m\}$ forms a partition of the set of all primes, and each sum $\sum_{p \in E_j} \frac{1}{p}$ diverges, as in Hal\'{a}sz' theorem. These results, moreover, yield various interesting arithmetic applications. Here are examples of applications of our Theorems \ref{THM1} and \ref{THM2}, respectively (which are stated in forms far weaker than what can actually be determined from our Theorems \ref{THM1} or \ref{THM2}).
\begin{thm} \label{THMAPP}
a) Let $E_1$ be the set of primes $p \equiv 3 \text{ (mod $10$)}$; $E_2$ the set of all primes not in $E_1$ with $l \geq e^{e^{100}}$ digits for which the first $\llf \frac{\log_3 l}{\log 10} \rrf$ of its digits are fixed; and let $E_0$ denote the complement of all of these sets.  Suppose $E_j(x)^2 \ll k_j \ll \log^{\frac{1}{2}} x$ for each $j$. Then there is some $\sg > 1$ depending on $x$ such that if $z := e^{\frac{1}{\sg-1}}$,
\begin{equation*}
\pi(x;\mbf{E},\mbf{k}) \sim x\frac{E_0\left(z\right)^{k_0}E_1\left(z\right)^{k_1}E_2\left(z\right)^{k_2}}{k_0!k_1!k_2!}e^{-(E_0\left(z\right) + E_1\left(z\right) + E_2\left(z\right))}G_1(\mbf{k},x)	
\end{equation*}
where $G_1(\mbf{k};x)$ is a well-determined function satisfying $G_1(\mbf{k};x) \gg \exp\left(k_0+k_1+k_2\right)$. \\
b) Let $q$ be prime, and let $E_1$ be the set of primes satisfying $p \equiv a \text{ (mod $q$)}$, where $a$ is a quadratic residue, and let $E_2$ be the analogous set but for $a$ a quadratic non-residue modulo $q$. Suppose $\log_2^2 x \ll k_1,k_2 \ll \log^{\frac{1}{2}} x$. Let $c_j$ be the constant such that $\sum_{p \in E_j \atop p \leq x} \frac{1}{p} = \frac{1}{2}\log_2 x + c_j + O\left(\frac{1}{\log x}\right)$, for $j = 1,2$. Then there is some $\sg > 1$ depending on $x$ such that if $z := e^{\frac{1}{\sg-1}}$
\begin{equation*}
\pi(x;(E_1,E_2),(k_1,k_2)) \sim \frac{x}{\log x}\left(\frac{\left(\frac{1}{2}\log_2 z + c_1\right)^{k_1}}{k_1!}\right)\left(\frac{\left(\frac{1}{2}\log_2 z + c_2\right)^{k_2}}{k_2!}\right)G_2(\mbf{k},x),
\end{equation*}
where $G_2(\mbf{k},x) := \sqrt{\frac{k_1+k_2}{2\log_2 z}}\exp\left((1+o(1))\frac{k_1+k_2}{\log_2 z} + g(\mbf{k},x)\right)$and $g(\mbf{k})$ is a well-determined function satisfying $g(\mbf{k},x) \gg k_1+k_2$.
\end{thm}
In part a), we note that since $\llf \frac{\log_3 l}{\log 10} \rrf \geq 2$, the units digit (i.e., residue class modulo 10) is independent of the fixed digits, so $E_2$ is well-defined. We can apply Theorem \ref{THM1} in this case once it is shown that $\sum_{p \leq t \atop p \in E_2} \frac{1}{p} \asymp \frac{\log_2 t}{\log_3 t}$.  We prove this in Appendix 1. The main feature of this example is that $E_2$ \emph{does not} have Dirichlet density.\\
Part b) is a direct consequence of Theorem \ref{THM2}, as precisely half of the non-zero residue classes modulo $q$ are quadratic residues, the other half being quadratic non-residues (and only one prime which congruent to zero modulo $q$, obviously). Hence, we use $m_j = \frac{1}{2}(q-1) = \frac{\phi(q)}{2}$ for $j =1,2$. \\
The forms of $G_1(\mbf{k},x)$ and $g(\mbf{k};x)$ can be determined explicitly, but they are too complicated to state here.
\section{Results and Strategy of Proof}
To determine an asymptotic formula, we complete the collection of sets of primes under consideration to form a partition of $\mc{P}$.  Precisely, let $E_1,\ldots,E_n \subseteq \mc{P}$ be disjoint, non-empty subsets of primes, and let $E_0 := \mc{P} \bk \left(\bigcup_{1 \leq i \leq n} E_i\right)$ (if $E_0$ is empty and the sets already form a partition then we relabel $E_j$ as $E_{j-1}$ for each $1 \leq j \leq n$).  Observe that if $\mbf{E} := (E_0, E_1,\ldots,E_n) := (E_0, \mbf{E}')$ and $\mbf{k} := (k_0,k_1,\ldots,k_n) := (k_0,\mbf{k}')$ then
\begin{equation*}
\pi(x; \mbf{E}',\mbf{k}') = \sum_{0 \leq k_0 \leq \llf \frac{\log x}{\log_2 x}\rrf} \pi\left(x;\mbf{E},(k_0,\mbf{k'})\right).
\end{equation*}
Therefore, having asymptotic formulae for $\pi(x; \mbf{E},\mbf{k})$ with $k_0$ not too small or not too large (in a precise sense) will suffice to provide asymptotic formulae for $\pi(x;\mbf{E}',\mbf{k}')$. In this paper, we shall only consider determining $\pi(x;\mbf{E},\mbf{k})$ in a range of the  parameters $\mbf{k}$. Therefore, for the remainder of this paper we deal solely with a vector of disjoint, non-empty sets $\mbf{E}$ whose union is all of $\mc{P}$.\\
Let $\mbf{z} := (z_0,\ldots,z_n) \in \mb{C}^{n+1}$. We define the Euler product
\begin{equation*}
F(\mathbf{z};s) := \prod_{0 \leq j \leq n} \prod_{p \in E_j} \left(1+\frac{z_j}{p^s -1}\right),
\end{equation*}
and we observe that Perron's formula (see II.2 in \cite{Ten2}) yields the multivariable generating function
\begin{equation} \label{PERRFORM}
\frac{1}{2\pi i} \int_{(\sg)} F(\mathbf{z};s) x^s \frac{ds}{s} = \sum_{m \leq x} \prod_{0 \leq i \leq n} z_i^{\omega_{E_i}(m)} =: \sum_{\mathbf{k} \in \mb{N}_0^{n+1}} \pi(x; \mathbf{E},\mathbf{k}) z_0^{k_0}\cdots z_n^{k_n},
\end{equation}
where $\sg > 1$ is fixed, and $\int_{(\sg)}$ denotes the path integral along the vertical line $\{s \in \mb{C} : \text{Re}(s) = \sg\}$.
By applying an $(n+1)$-dimensional version of Cauchy's theorem, we have
\begin{equation} \label{PerrId}
\pi(x;\mathbf{E},\mathbf{k}) = \frac{x^{\sigma}\left(\prod_{0 \leq i \leq n} \rho_i^{-k_i}\right)}{(2\pi)^{n+2}}  \int_{[-\pi,\pi]^{n+1}}d\mbf{t}e^{-i \mbf{k}\cdot \mbf{t}}\int_{(\sigma)} F(\mbf{\rho}e^{i\mbf{t}},\sg+i\tau) x^{i\tau} \frac{d\tau}{\sg+i\tau},
\end{equation}
where, for each $0 \leq j \leq n$, $\rho_j$ is a parameter at our disposal (since \eqref{PERRFORM} is a polynomial for any fixed $x \geq 2$, it follows that the only possible pole in each variable is $z_j = 0$ and we may select $\rho_j$ as we please), and $\mbf{\rho}e^{i\mbf{t}} := (\rho_0e^{it_0},\ldots,\rho_ne^{it_n})$. This integral may, in principle, be computed via the Saddle Point method, which we now describe.  Assuming that we have $F(\mathbf{z};s) \neq 0$ on some open, simply connected subset of $\mb{C}^{n+2}$, then by the inverse function theorem we can find a holomorphic branch of logarithm for $F$, i.e., we can find a holomorphic function $f(\mathbf{z};s)$ such that $F = e^f$ on this set.  If we can find a \emph{unique} solution vector $(\mbf{\rho},\sigma)$, with real entries, to the simultaneous equations $\frac{\partial f}{\partial z_j}(\mbf{\rho},\sigma) = \frac{k_j}{\rho_j}$ and $\frac{\partial f}{\partial s}(\mathbf{\rho},\sigma) = -\log x$ then a rapidly oscillating contribution to the imaginary part in the Taylor expansion of $f$ with respect to $\mbf{z}$ and $s$ near $(\mbf{\rho},\sg)$ (see \eqref{hDef2}) is annihilated near this point. One therefore expects that the entire integral is approximated by the integrand in a small neighbourhood of this solution, as, outside this interval, the oscillating argument is responsible for significant cancellations and therefore lower-order contributions.  The existence and uniqueness of $(\mbf{\rho},\sg)$ is guaranteed as a consequence of estimates related to the prime number theorem (see Lemma \ref{LemEU}). \\
Let us assume \textbf{one} of the following hypotheses, for $\sg > 1$ fixed: \\
($\mathbf{H}_1(\sg)$):  There exists an index $j_1$ such that $E_{j_1}\left(e^{\frac{1}{\sg-1}}\right) \gg_n \log\left(\frac{1}{\sg-1}\right)$ \emph{and} 
\begin{equation*}
\sum_{p \in E_{j_1} \atop p \leq e^{\frac{1}{\sg-1}}} \frac{\log p}{p} \gg_n \frac{1}{\sg-1}.
\end{equation*}
($\mathbf{H}_2(\sg)$): Let $j''$ be the index maximizing $(\sg-1)^3\frac{k_j}{E_j\left(e^{\frac{1}{\sg-1}}\right)}\sum_{p \in E_j} \frac{\log^3 p}{p^{\sg}}$. Then 
\begin{equation*}
(\sg-1)^3 \sum_{p \in E_{j''}} \frac{\log^3 p}{p^{\sg}} \ll_n (\sg-1)^2\sum_{p \in E_{j''}} \frac{\log^2 p}{p^{\sg}}. 
\end{equation*}
($\mathbf{H}_3(\sg)$): Let $j''$ be as in hypothesis $\textbf{H}_2(\sg)$. Let $\sg' := 1+\frac{1}{2}(\sg-1)$. Then 
\begin{equation*}
\sum_{p \in E_{j''} \atop p > e^{\frac{1}{\sg'-1}}} \frac{\log^2 p}{p^{\sg'}} \ll_n \sum_{p \in E_{j''}} \frac{\log^2 p}{p^{\sg}}. 
\end{equation*}
By the pigeonhole principle, there is an index $i_1$ such that $E_{i_1}(x) \gg_n \log_2 x$, as the functions $E_j(x)$ sum to a function asymptotic to $\log_2 x$ (as in \eqref{MERTENS} below).  Similarly, there is an index $i_2$ such that $\sum_{p \in E_{i_2} \atop p\leq x} \gg_n \log x$. In natural cases, such as that of arithmetic progressions in Theorem 2 below, we \emph{can} find a common index so that both lower bounds hold for a given set.  However, the assertion that we can choose $i_1 = i_2$ is false in general.  In Appendix 2, we construct an example of a partition of $\mc{P}$ for which $\mathbf{H}_1(\sg)$ does not hold for infinitely many choices of $x$, and thus of $\sg$.  \\
We will furthermore show in Lemma \ref{RATIOS} below that $\mathbf{H}_3(\sg)$ implies $\mathbf{H}_2(\sg)$, though the latter is sufficient for us.  We give an example in Remark \ref{REM1} that $\mathbf{H}_2(\sg)$ fails in general, but this example is not actually applicable to our problem. The need for all of these hypotheses is in order to sort out an error term in Lemma \ref{LemInt} below. 
It is not clear that any of these hypotheses is necessary to prove Theorem \ref{THM1}.
Given $t \geq 2$, let $E_j(t) := \sum_{p \in E_j \atop p \leq t} \frac{1}{p}$. Furthermore, write $M> 0$ to be the constant such that
\begin{equation}
\sum_{p \leq t} \frac{1}{p}-\log_2 t = M+O\left(\frac{1}{\log t}\right), \label{MERTENS}
\end{equation}
by Mertens' theorem (see I.1 of \cite{Ten2}).  Also, for each $j$, let $k_j^{\max}$ denote the maximal values of $\omega_{E_j}(n)$ with $n \leq x$. \\
We shall prove the following theorem, which contradicts the heuristic \eqref{HEUR} for many choices of the vector $\mbf{k}$.
\begin{thm} \label{THM1}
Let $x \geq 3$ be sufficiently large, and suppose $\mbf{E}$ is a vector of $n+1$ sets of primes partitioning $\mc{P}$, such that $E_j(x) \ra \infty$ as $x \ra \infty$ for each $j$. Let $\mu,\e > 0$ be fixed but arbitrary, and let $A \geq \frac{3}{2}$. Let $y(x) := e^{\log^{\frac{1}{3}}x}$.  
Let $m_j \geq 2$ and suppose that $E_{j}(x)^{m_j} \ll k_{j} \ll \min\{k_j^{\max}, (\log x)^{\frac{2}{3}-\mu}\}$ for each $0 \leq j \leq n$. Then there exists a vector of parameters $(\mbf{\rho},\sg) \in (0,\infty)^{n+2}$ satisfying
\begin{align*}
\left(\sum_{0 \leq j \leq n} \frac{1}{\rho_j^2}\right)^{-\frac{1}{2}}(1+o(1)) \leq (\sg-1)\log x \leq  \left(\sum_{0 \leq j \leq n} \rho_j^2\right)^{\frac{1}{2}}(1+o(1)) \\
\rho_j = \frac{k_j}{E_j\left(e^{\frac{1}{\sg-1}}\right)} \left(1+O\left(\frac{E_j\left(k_j/E_j\left(y(x)\right)\right)}{E_j\left(y(x)\right)}\right)\right),
\end{align*}
such that if we assume one of hypotheses $\mathbf{H}_1(\sg)$, $\mathbf{H}_2(\sg)$ or $\mathbf{H}_3(\sg)$ is true then the following holds. Set $\eta_j := k_j^{-1}\rho_jE_j\left(e^{\frac{1}{\sg-1}}\right) - 1$, $\alpha_j := (\sg-1){\sum_{p \in E_j} \frac{\log p}{p^{\sg}}}$ and $\beta_j := (\sg-1)^2\sum_{p \in E_j} \frac{\log^2 p}{p^{\sg}}$ for each $j$, and let $z := e^{\frac{1}{\sg-1}}$. Then
\begin{equation*}
\pi(x;\mbf{E},\mbf{k}) = x \left(\prod_{0 \leq j \leq n} \left(1+O_{\e,\mu}\left(E_j\left(y(x)\right)^{-\frac{1}{2}+\e}\right)\right)\frac{E_j\left(z\right)^{k_j}}{k_j!}e^{-E_j\left(z\right)}\right) \mc{F}(\mbf{k},\sg) + O\left(x\frac{F(\mbf{\rho},\sg)}{\log^A x}\right),
\end{equation*}
where, for $\mc{R}$ implicitly defined by $e^{\mc{R}} := F(\mbf{\rho};\sg)e^{-\sum_{0 \leq j \leq n} k_j(1+\log(1+\eta_j))}$,
\begin{equation*}
\mc{F}(\mbf{k},\sg) := \left(2\pi\sum_{0 \leq j \leq n} \rho_j\beta_j\right)^{-\frac{1}{2}}\exp\left(\sum_{0 \leq j \leq n} \rho_j\alpha_j\left(1+O\left(E_j(x)^{-1}\right)\right) + \mc{R} +M\right).
\end{equation*}
Moreover, there exists a constant $\phi > \frac{1}{4}$ such that $\mc{R} \geq \sum_{0 \leq j \leq n} k_j(\phi-\log\left(1+\eta_j\right))$. \\
When $\mathbf{H}_1(\sg)$ holds then we can take $m_j = 2$ when $j \neq j'$, and $m_{j'} = 4$. When either of the other two hypotheses hold then we can take $m_j = 2$ for each $j$.
\end{thm}
Note that by the estimate on $\rho_j$ above, $\eta_j = o(1)$, so $\mc{R}$ provides exponential growth to the factor $\mc{F}(\mbf{k},\sg)$. We remark that while the error term is not satisfactory in the sense that we cannot get a power of $E_j(x)$ explicitly, in many cases the distinction between $E_j\left(y(x)\right)$ and $E_j(x)$ is negligible (for instance, if $E_j(t) \asymp \log_k^C t$, where $\log_k$ is the $k$th iterated logarithm with $k \geq 2$, and $C > 0$, uniformly over a sufficiently large interval, then the ratio $\frac{E_j\left(y(x)\right)}{E_j(x)}$ is bounded).\\
Determining precise estimates for the parameter vector $(\mbf{\rho},\sg)$ is somewhat complicated in general. In certain cases of interest, however, their form is much simpler, specifically, when the ratios $\alpha_j$ can be explicitly calculated. In particular, suppose each $E_j$ is a union of $m_j$ arithmetic progressions modulo $q_j$, for $1 \leq j \leq n$. Let $q := \text{lcm}\left\{q_1,\ldots,q_m\right\}$, and define $E_0 := \mc{P} \bk \left(\bigcup_{1 \leq j \leq n} E_j\right)$, as before, so that each set $E_j$ is a union of arithmetic progressions modulo $q$. It follows from \cite{Will} that, for a given arithmetic progression $S$ modulo $q$ there exists some well-determined constant $c$ such that $S(t) = \frac{1}{\phi(q)}\log_2 t + c + O\left(\frac{1}{\log t}\right)$, so we can write $E_j(t) = \frac{m_j}{\phi(q)}\log_2t + c_j + O\left(\frac{1}{\log t}\right)$, where $c_j$ is the sum of all the constants coming from the $m_j$ arithmetic progressions in $E_j$. Determining the parameters $(\mbf{\rho},\sg)$ as in Lemma \ref{LemDirDens}, we deduce immediately the following theorem from the previous one, which is an extension of a result of Delange (a corollary of a special case of Theorem C in section 6.4 of \cite{D}), who proved a similar result for a set of parameters $k_j \ll E_j(x)$.
\begin{thm} \label{THM2}
Let $x \geq 3$ be sufficiently large, and let $A \geq \frac{3}{2}$. Let $E_0,\ldots,E_n$ be unions of arithmetic progressions modulo $q$, as described above, consisting of $m_j$ residue classes for each $j$. Set $K := \sum_{0 \leq j \leq n} k_j$.
Then, provided that $E_j(x)^{2} \ll k_j \ll \log^{\frac{3}{2}-\mu} x$ for each $j$ then there exist constants $\sg$ and $\mbf{\rho}$ such that  
\begin{align*}
\sg-1 &= \frac{K}{\log x \log(\log x/K)}\left(1+O\left(\frac{1}{\log(\log x/K)}\left(\phi(q)\sum_{0 \leq j \leq n} \frac{1}{m_j} + \log_3 x\right)\right)\right) \\
\rho_j &= \frac{\phi(q) k_j}{m_j \log\left(\frac{\log x}{K}\log(\log x/K)\right)}\left(1+O\left(\frac{\log(\phi(q)k_j/m_j\log(\log x))}{\log_2 x}\right)\right),
\end{align*}
such that the following holds. Set $M(\mbf{\rho}):= \left(\frac{1}{2\pi}\sum_{0 \leq j \leq n} \frac{\rho_j m_j}{\phi(q)}\right)^{\frac{1}{2}}$, write $\psi_j := \left(\frac{\phi(q)}{m_j \log_2 x}\right)^{\frac{1}{2}-\e}$ and let $\eta_j$ be defined as in Theorem \ref{THM1}. Then
\begin{equation*}
\pi(x;\mbf{E},\mbf{k}) = \frac{x}{\log x}\left(\prod_{0 \leq j \leq n} \left(1+O\left(\psi_j\right)\right)\frac{\left(\frac{m_j}{\phi(q)}\log\left(\frac{1}{\sg-1}\right) + c_j\right)^{k_j}}{k_j!}\right) \mc{F}(\mbf{k},\sg) + O\left(x\frac{F(\mbf{\rho},\sg)}{\log^A x}\right),
\end{equation*}
where, for $e^{\mc{R}} := F(\mbf{\rho};\sg)e^{-\sum_{0 \leq j \leq n} k_j(1+\log(1+\eta_j))}$, 
\begin{equation*}
\mc{F}(\mbf{k},\sg) = M(\mbf{\rho})\exp\left((\sg-1)\log x + \mc{R} +M\right). 
\end{equation*}
Moreover, there exists $\phi > \frac{1}{4}$ such that $\mc{R} \geq \sum_{0 \leq j \leq n} k_j(\phi-\log\left(1+\eta_j\right))$. 
\end{thm}
Because for each $j = 1,2,3$, $\mathbf{H}_j(\sg)$ trivially holds for collections of arithmetic progressions, there is no need to mention them in Theorem \ref{THM2}. \\
The factor $\mc{F}(\mbf{\rho},\sg)$ skews $\pi(x;\mbf{E},\mbf{k})$ from Poisson-like behaviour in the range $E_j(x)^{m_j} \ll k_j \ll \log^{\frac{2}{3}-\mu} x$. In a forthcoming paper \cite{next}, however, we will show that the form of this factor is rather different for $E_j(x)^{\e} \ll k_j \ll E_j(x)$, so that $\pi(x;\mbf{E},\mbf{k})$ resembles \eqref{SelSat}. \\
To ensure that there exists a domain where we can write $F = e^f$, we must check that, given $\rho_j$ and $\sg$, we have $p^s - 1 + z_j \neq 0$ for each prime $p$.  Let us show that there exists a symmetric box in $\mb{R}^{n+2}$ of the shape $[-T,T] \times \prod_{0 \leq j \leq n} [-\theta_j,\theta_j]$, for parameters $T,\theta_0,\ldots, \theta_n > 0$. For each $j$ and $p\in E_j$, either $\rho_j^{\frac{1}{\sg}}$ is smaller than the smallest prime in $E_j$, or we may choose $p_j \in E_j$ minimally such that $p_{j-1}^{\sg} < \rho_j < p_j^{\sigma} - 1$ (we can assume a strict inequality on either side by skewing the value of $\rho_j$ slightly; moreover, it is easy to see that $|p_j^{\sg}-p_{j-1}^{\sg}| > 2$ for any $\sg > 1$, so the interval here is non-trivial). It will be clear from Lemma \ref{PARAMS}, in any case, that if $0 < p^{\sg}-\rho_j < 1$ for some prime $p$ we can replace $\rho_j$ by $\rho_j - 1$ in case $k_j \gg E_j(x)^{2}$, so that our choice of $p_j$ is well-defined. Suppose that $T < \frac{\pi}{6 |\log \rho_j|}$. Then if $p \geq p_j$, we have
\begin{equation*}
|p^s-1| \geq p^{\sg}-1 \geq p_j^{\sg}-1 > \rho_j = |z_j|,
\end{equation*}
so that $\left|1+\frac{z_j}{p^s-1}\right| \geq 1-\left|\frac{z_j}{p^s-1}\right| > 0$.  On the other hand, if $\rho_j > p_{j-1}^{\frac{1}{\sg}}$, when $p \leq p_{j-1}$, as $\cos(\tau \log p) \geq \cos(|\tau| \log p_{j-1}) \geq \cos(T\log \rho_j) \geq \frac{\sqrt{3}}{2}$,
\begin{align*}
|p^s-1|^2 &= p^{2\sg} + 1-2p^{\sg}\cos(\tau \log p) = (p^{\sg}-1)^2 + 2p^{\sg}(1-\cos(\tau \log p)) \\
&\leq (\rho_j-1)^2 + (2-\sqrt{3})\rho_j \leq \rho_j^2 +1-\sqrt{3}\rho_j < \rho_j^2 = |z_j|^2,
\end{align*}
and, as such, $\left|\frac{z_j}{p^s-1}\right| > 1$. We shall see that $\rho_j \ll k_j \ll \frac{\log x}{\log_2 x}$, the last estimate uniform in $0\leq j \leq n$, so that $|\tau| \ll \frac{1}{\log_2 x}$ will be sufficient.  We shall, in fact, consider $T \ll \sg-1$ which, by Lemma \ref{PARAMS}, ensures that the above condition is satisfied, since $(\sg-1) \ll \log^{-\frac{1}{3}} x$.\\
In what follows, for a function $g$ in several variables, among which $u$, write $g_u$ to denote the partial derivative of $g$ with respect to $u$. Suppose that the identity $F(\mbf{z};s) = e^{f(\mbf{z};s)}$ is well-defined for $|\text{arg}(z_j)| \leq \theta_j$ and $|s-\sg| \leq T$, where $(\rho;\sg)$ is a point where $f_s(\mbf{\rho};\sg) = -\log x$ and $f_{z_j}(\mbf{\rho};\sg) = \frac{k_j}{\rho_j}$. In this region, we  can determine a Taylor series expansion for $f$ locally around our chosen point.  In particular, upon expanding first in $\tau$ and then in $t_j$,
\begin{align}
f(\mbf{z};s) &= f(\mbf{z};\sg) + i\tau f_s(\mbf{z};\sg) -\int_0^{\tau} d\tau' (\tau-\tau') f_{ss}(\mbf{z};\sg +i\tau) \\ \nonumber
&= f(\mbf{\rho};\sg) + \sum_{0 \leq j \leq n} \rho_j(e^{it_j}-1) f_{z_j}(\rho_j;\sg) + \sum_{0 \leq j \leq n} \int_{\rho_j}^{z_j} dw_j (z_j-w_j) f_{z_jz_j}(w_j;\sg) \nonumber\\
&+i\tau f_s(\rho;\sg) + i\tau\sum_{0 \leq j \leq n}\int_{\rho_j}^{z_j}dw_jf_{z_js}(\rho;\sg) \label{hDef2}\\
&=: f(\rho;\sg) + \sum_{0 \leq j \leq n} k_j(e^{it_j}-1) -i\tau \log x + h(\mbf{z};s) \label{hDef},
\end{align}
(we have slightly abused notation and written $f_{z_j}$ as a function of $z_j$, rather than the entire vector $\mbf{z}$, but this is legitimate given that, from \eqref{Zit}, $f_{z_j}$ has no dependence on $z_{j'}$ for any $j \neq j'$). For $\theta_j$ sufficiently small and $|t_j| \leq \theta_j$,
\begin{align*}
\sum_{0\leq j \leq n} k_j(e^{it_j}-1) &= \sum_{0 \leq j \leq n} k_j\left(it_j - \frac{1}{2}t_j^2 + O\left(\theta_j^3\right)\right) \\
&= \sum_{0\leq j \leq n} ik_jt_j - \frac{1}{2}\sum_{0 \leq j\leq n} k_jt_j^2 + O\left(\sum_{0\leq j \leq n} k_j\theta_j^3\right).
\end{align*}
Suppose further that we can truncate the integrals in \eqref{PerrId} to the box $B := [-T,T]\times \prod_{0 \leq j \leq n} [-\theta_j,\theta_j]$, for parameters $T,\theta_0,\ldots,\theta_n > 0$, with error term $R = R(T,\theta_0,\ldots,\theta_n)$ associated with this truncation.  Then, incorporating the defining conditions for $\mbf{\rho}$ and $\sg$ and the above-mentioned Taylor expansion into \eqref{PerrId}, we have
\begin{equation} \label{NEWPERR}
\pi(x;\mbf{E},\mbf{k}) = \frac{1}{(2\pi)^{n+2}}\frac{x^{\sg}F(\rho;\sg)}{\rho_0^{k_0}\cdots \rho_n^{k_n}} \int_B \frac{d\tau}{\sg + i\tau} dt_0\cdots dt_n e^{-\frac{1}{2}\sum_{0\leq j \leq n} k_jt_j^2}H(\mbf{\rho}e^{i\mbf{t}};\sg + i\tau) + R
\end{equation}
where $H(\mbf{z};s) := \exp\left(h(\mbf{z};s) + O\left(\sum_{0 \leq j \leq n} k_jt_j^3\right)\right)$. The truncated integral here is simplified once $H(\mbf{z};s)$ is computed effectively.  This is done in Lemma \ref{ERR}. In Lemma \ref{LemInt}, the main term of the integral is computed in a straightforward manner, and specific values of $\theta_0,\ldots,\theta_n$ and $T$ are chosen based on suitable estimates for the partial derivatives of $f$ with respect to $s$ and each of the $z_j$. Given these choices, it is shown in Lemma \ref{LemAppTrunc} that $R(T,\theta_0,\ldots,\theta_n)$ is small. The major remaining ingredient, otherwise, is to show that such a truncation can be made.  This is dealt with at the beginning of Section 5. \\
Finally, we should remark that we have not made any attempt to determine asymptotic formulae that are also uniform in the number $n+1$ of sets $E_0,\ldots,E_n$.  In particular, throughout this paper, all estimates may depend on $n$.
\section{Preparatory Lemmata}
By definition, $f(\mbf{z};s) = \sum_{0 \leq j \leq n} \sum_{p \in E_j} \log\left(1+\frac{z_j}{p^s - 1}\right)$. 
As mentioned in the previous section, for each tuple $(z_0,\ldots,z_n) \in \mb{C}^{n+1}$ and $s \in \mb{C}$ with $\text{Re}(s) > 1$ such that $1+\frac{z_j}{p^s-1} \neq 0$ for each $p \in E_j$, $f(\mbf{z};s)$ is holomorphic, hence $\mb{C}$-differentiable with respect to each of its variables. Moreover, the series defining $f(\mbf{z};s)$ is absolutely and uniformly convergent on compact subsets of its domain (since, for $\sg > 1$ the sum $\sum_{k \geq 1} \sum_{p} \frac{1}{kp^{ks}}$ converges absolutely). Now, let $g_j(s) := \sum_{p \in E_j} \frac{1}{p^s}$, valid uniformly in $\sigma > 1$ for each $j$. Then routine calculations show that
\begin{align}
f_{z_j}(\mbf{z};s) 
&= g_j(s) - (z_j - 1) \sum_{p \in E_j} \frac{1}{p^s(p^s-1+z_j)}; \label{Zit}\\	
f_{z_jz_j}(\mbf{z};s) &= -\sum_{p \in E_j} \frac{1}{(p^s - 1 + z_j)^2} \nonumber\\
&= g_j(s) - \sum_{p \in E_j} \frac{1}{p^s(p^s-1+z_j)} \left((2z_j-1) - \frac{z_j(z_j-1)}{p^s-1+z_j}\right); \label{ZiZi}\\
f_s(\mbf{z};s) &= -\sum_{0 \leq j \leq n} z_j \sum_{p \in E_j} \frac{p^s \log p}{(p^s-1)(p^s-1+z_j)}; \\
f_{ss}(\mbf{z};s) &= \sum_{0 \leq j \leq n} z_j\sum_{p \in E_j} \frac{p^s \log^2 p}{(p^s-1 + z_j)(p^s-1)} \left(\frac{1}{p^s-1+ z_j}\left(1 + \frac{1}{p^s-1}\right) + \frac{1}{p^s-1}\right); \\
f_{sz_j}(\mbf{z};s) &= -\sum_{p \in E_j} \frac{p^s \log p}{(p^s-1 + z_j)^2}. \label{ZiS}
\end{align}
Specific estimates for some of these functions will be determined in what follows (for instance, in Lemma \ref{ERR}). \\
Because the sets $E_j$ partition $\mc{P}$, it follows immediately from the pigeonhole principle that for each $y$ there are indices $j_0,j_1,j_2$ such that $\sum_{p \in E_{j_l} \atop p \leq y} \frac{\log^l p}{p} \geq \frac{1}{n+1}\sum_{p \leq y} \frac{\log^l p}{p}$ for each $l \in \left\{0,1,2\right\}$. We will use this fact repeatedly throughout this paper, by virtue of the following lemma connecting $g_j(\sg)$ to $E_j(t)$ for certain choices of $t$.
\begin{lem} \label{LemUnif}
Let $z \geq 3$ and suppose $\sigma = 1 + \frac{\eta}{\log z}$, where $\eta > 0$. Then for any $\mc{P}' \subseteq \mc{P}$, we have
\begin{align*}
\sum_{p \leq z \atop p \in \mc{P}'} \frac{1}{p^{\sigma}} &= \sum_{p \leq z \atop p \in \mc{P}'} \frac{1}{p} + R(\eta)+ O\left(\frac{1}{\log z}\right),
\end{align*}
where $0 \leq R(\eta) \leq \eta^{-1}(e^{\eta}-1-\eta)$, with equality in the upper bound holding when $\mc{P}' = \mc{P}$. Moreover, $\sum_{p \in \mc{P}' \atop p> z} \frac{1}{p^{\sg}} = O(1/\eta)$, the constant being uniform in the choice of set $\mc{P}'$.
\end{lem}
\begin{proof}
From the Taylor expansion $e^c-1 = \sum_{l \geq 1} \frac{c^l}{l!}$, taking $c := \eta\frac{\log p}{\log z}$ for each $p \leq z$, we have
\begin{equation} \label{ESTHERE} 
\sum_{p \leq z \atop p \in \mc{P}'} \frac{1}{p} - \sum_{p \leq z \atop p \in \mc{P}'} \frac{1}{p^{\sg}} = \sum_{p \leq z \atop p \in \mc{P}'} \frac{p^{\sigma - 1}-1}{p} = \sum_{l \geq 1} \frac{\eta^l}{(\log z)^l l!}\sum_{p \leq z \atop p \in \mc{P}'} \frac{\log^{l} p}{p} =: R(\eta).
\end{equation}
Clearly, $R(\eta) > 0$. By partial summation with the Prime Number Theorem, we have
\begin{equation} \label{REMHERE}
\sum_{p \leq z} \frac{\log^l p}{p} = \int_{2^-}^z \log^{l} t \frac{dt}{\log t} + O\left(\int_{2^-}^z \log^{l-1} t \frac{dt}{\log t}\right) = \frac{1}{l} \log^{l} z + O(\log^{l-1} z),
\end{equation}
so $R(\eta) \leq \sum_{l \geq 1} \frac{\eta^l}{(l+1)!} = \eta^{-1}(e^{\eta}-1-\eta)$, as claimed. Inserting \eqref{REMHERE} into \eqref{ESTHERE} proves the first statement.\\
For the second assertion, we clearly have, again by partial summation with the prime number theorem
\begin{align*}
\sum_{p \in \mc{P}' \atop p > z} \frac{1}{p^{\sg}} &\leq 2\int_{z}^{\infty} e^{-\eta\frac{\log t}{\log z}}\frac{dt}{t\log t} = 2\int_{\eta}^{\infty} e^{-u}\frac{du}{u} \leq \frac{2}{\eta}\int_{0}^{\infty} e^{-u}du = \frac{2}{\eta}.
\end{align*}
\end{proof}
We shall frequently need the following standard facts, which we prove for completeness.
\begin{lem} \label{LAURENTZETA}
i) Given $0 < \sg-1 < 1$, we have $\sum_{p} \frac{\log p}{p^{\sg}} = \frac{1}{\sg-1} + O(1)$ and $\sum_{p} \frac{\log^2 p}{p^{\sg}} = \frac{1}{(\sg-1)^2} + O(1)$. \\
ii) For $q \geq 2$ and $0 \leq a \leq q-1$, $\sum_{p \equiv a \text{ (mod $q$)}} \frac{\log p}{p^{\sg}} = \frac{1}{\phi(q)}\frac{1}{\sg-1} + O(1)$ and $\sum_{p \equiv a \text{ (mod $q$)}} \frac{\log p}{p^{\sg}} = \frac{1}{\phi(q)}\frac{1}{(\sg-1)^2} + O(1)$.
\end{lem}
\begin{proof}
i) From II.3 of \cite{Ten2}, $\zeta(\sg) = \frac{1}{\sg-1} + h_1(\sg)$, where $h_1(\sg)$ is differentiable in $\sg$ at $\sg = 1$. It follows that
\begin{equation*}
\log\left(\zeta(\sg)\right) = -\sum_p \log\left(1-p^{-\sg}\right) = \sum_p \frac{1}{p^{\sg}} + h_2(\sg)
\end{equation*}
for $h_2$ a continuously differentiable function in $\sg$ at $\sg = 1$. Differentiating with respect to $\sg$, we see that
\begin{equation*}
\sum_p \frac{\log p}{p^{\sg}} = -\frac{\zeta'}{\zeta}(\sg) -h_2'(\sg),
\end{equation*}
and since $h_2'$ is bounded near $\sg = 1$ and $\frac{\zeta'}{\zeta}(\sg) = -\frac{1}{\sg-1} + (h_1'(\sg) - h_2'(\sg))(\sg-1)$, the first conclusion follows. The second one follows from the first by differentiation with respect to $\sg$.\\
ii) By orthogonality of Dirichlet characters mod $q$,
\begin{equation*}
\sum_{p \equiv a \text{ (mod $q$)}} \frac{\log p}{p^{\sg}} = \frac{1}{\phi(q)}\sum_{\chi} \overline{\chi}(a)\sum_p \frac{\chi(p)\log p}{p^{\sg}} = -\frac{1}{\phi(q)}\sum_{\chi} \overline{\chi}(a) \left(\frac{L'(\sg,\chi)}{L(\sg,\chi)}+O(1)\right),
\end{equation*}
the last estimate following in a manner similar to i), with $L(s,\chi) := \sum_{n \geq 1} \chi(n)n^{-s}$ for $\text{Re}(s) > 1$. It is well-known (see II.8 of \cite{Ten2}) that the Dirichlet series for $L(s,\chi)$ is well-defined and holomorphic for $\text{Re}(s) > 0$ whenever $\chi$ is not the trivial character $\chi_0$. Hence, $\frac{L'(\sg,\chi)}{L(\sg,\chi)} = O(1)$ for $\chi \neq \chi_0$, and thus,
\begin{equation*}
\sum_{p \equiv a \text{ (mod $q$)}} \frac{\log p}{p^{\sg}} = -\frac{1}{\phi(q)}\left(\frac{L'(\sg,\chi_0)}{L(\sg,\chi_0)} + O(1)\right) + O(1).
\end{equation*}
Since $\chi(p) = 0$ if $p | q$ and is equal to 1 otherwise, 
\begin{equation*}
-\frac{L'(\sg,\chi_0)}{L(\sg,\chi_0)} = \sum_{p \nmid q} \frac{\log p}{p^{\sg}} + O(1) = \sum_{p} \frac{\log p}{p^{\sg}} + O\left(\log q\right) = -\frac{\zeta'}{\zeta}(\sg) + O\left(\log q\right).
\end{equation*}
Since $\phi(q) \gg q^{1-\e}$ for any $\e > 0$ (see I.5 in \cite{Ten2}), we have by i),
\begin{equation*}
\sum_{p \equiv a \text{ (mod $q$)}} \frac{\log p}{p^{\sg}} = \frac{1}{\phi(q)}\left(\frac{1}{\sg-1} + O\left(\log q\right)\right) + O(1) = \frac{1}{\phi(q)}\frac{1}{\sg-1} + O(1).
\end{equation*}
The second estimate follows similarly, using $\frac{d}{d\sg}\left(\frac{L'(\sg,\chi)}{L(\sg,\chi)}\right)$ in place of $\frac{L'(\sg,\chi)}{L(\sg,\chi)}$.
\end{proof}
We shall also require the following standard technical results first, whose proofs are also included for completeness.
\begin{lem} \label{SIMPLE}
Let $z > 1$ and $\sg > 1$. Then the following estimates hold:
\begin{align}
&\sum_p \frac{\log^l p}{p^{\sg}} \asymp (l-1)!(\sg-1)^{-l} \ \forall l \in \N \label{POWERS};\\
&\sum_{p > z} \frac{\log^l p}{p^{\sg}} \asymp \sum_{0 \leq j \leq l-1} \frac{(l-1)!}{j!} \frac{\log^j z}{(\sg-1)^{l+1-j}}; \label{LOG}\\
&\sum_{p > z} \frac{1}{p^{l\sg}} \ll \frac{1}{l(1+z)^{l\sg-1} \log (z+1)} \label{P2};\\
&\sum_{p > z} \frac{\log^l p}{p^{2\sg}} \ll \frac{\log^{l-1}(1+z)}{z^{2\sg-1}} \label{L2}.
\end{align}
\end{lem}
\begin{proof}
All of these statements follow by partial summation from the Prime Number Theorem. Indeed,
\begin{align*}
\sum_p \frac{\log^l p}{p^{\sg}} &\asymp \int_2^{\infty} \log^l t e^{-(\sg-1)\log t}\frac{dt}{t\log t}  = \int_{\log 2}^{\infty} u^{l-1} e^{-(\sg-1)u} du \\
&\asymp \frac{1}{(\sg-1)^l}\int_0^{\infty}u^{l-1}e^{-u}du = \Gamma(l)(\sg-1)^{-l}.
\end{align*}
Given that $\Gamma(l) = (l-1)!$ for $l \in \N$, the first result follows immediately. For the second,
\begin{align*}
\sum_{p > z} \frac{\log^l p}{p^{\sigma}} &= \left(1+O\left(\frac{1}{\log z}\right)\right) \int_z^{\infty} \frac{\log^l t}{t^{\sigma}} \frac{dt}{\log t} = \left(1+O\left(\frac{1}{\log z}\right)\right) \int_{\log z}^{\infty} du u^{l-1}e^{-(\sg-1)u} \\
&\asymp  \sum_{0 \leq j \leq l-1} \frac{(l-1)!\log^j z}{j!(\sg-1)^{l+1-j}}
\end{align*}
by repeated use of integration by parts. For the third,
\begin{equation*}
\sum_{p > z} \frac{1}{p^{l\sg}} \asymp \int_{\log z}^{\infty} \frac{du}{u} e^{-(l\sg-1)u}  \ll \frac{1}{\log z} \int_{\log z}^{\infty} e^{(1-l\sg)u} du \ll \frac{1}{lz^{l\sg-1} \log z}.
\end{equation*}
Finally, again by partial summation and integration by parts,
\begin{equation*}
\sum_{p > z} \frac{\log^l p}{p^{2\sg}} \ll \int_{\log z}^{\infty} u^{l-1} e^{-(2\sg - 1)u} du \ll l\frac{\log^{l-1}z}{(2\sg - 1)z},
\end{equation*}
arguing as in the proof of \eqref{LOG}, and noting that $2\sg-1 > 1$.
\end{proof}
In Section 5 in particular we will make use of the ratios
\begin{align*}
\alpha_j(Y) &:= (\sg-1)\sum_{p \in E_j \atop p \leq Y} \frac{\log p}{p^{\sg}}, \ \alpha_j := (\sg-1)\sum_{p \in E_j} \frac{\log p}{p^{\sg}}; \\
\beta_j(Y) &:= (\sg-1)\sum_{p \in E_j \atop p \leq Y} \frac{\log^2 p}{p^{\sg}}, \ \beta_j := (\sg-1)^2\sum_{p \in E_j} \frac{\log^2 p}{p^{\sg}} \\
\gamma_j(Y) &:= (\sg-1)\sum_{p \in E_j \atop p \leq Y} \frac{\log^3 p}{p^{\sg}}, \ \gamma_j := \frac{1}{2}(\sg-1)^3\sum_{p \in E_j} \frac{\log^3 p}{p^{\sg}}
\end{align*}
defined for each $0 \leq j \leq n$ and $Y \geq 2$. We will also write $\alpha_{j,\sg}$, $\beta_{j,\sg}$ or $\gamma_{j,\sg}$ to denote the dependence on $\sg$. Clearly, $\alpha_j = \lim_{Y \ra \infty} \alpha_j(Y)$, the same being true of $\beta_j$ and $\gamma_j$ as limits of $\beta_j(Y)$ and $\gamma_j(Y)$, respectively. Since the sets $E_j$ partition $\mc{P}$, the sum over all $j$ of the $\alpha_j$, or the $\beta_j$ or of the $\gamma_j$ are all asymptotically 1 by Lemma \ref{LAURENTZETA}.  We will require the following relationships between these ratios.
\begin{lem} \label{RATIOS}
Let $\sg > 1$ and define $\sg' := 1+\frac{1}{2}(\sg-1)$. For each $j$, $\gamma_j \ll \beta_j \log\left(\frac{1}{\beta_j}\right) \ll \alpha_j\log^2\left(\frac{1}{\alpha_j}\right)$; also, $\gamma_{j,\sg} \ll \beta_{j,\sg} + \beta_{j,\sg'}$.  In particular, hypothesis $\mathbf{H}_3(\sg)$ implies $\mathbf{H}_2(\sg)$. Moreover, $\alpha_j(Y) \ll \beta_j^{\frac{1}{2}}(Y)$, for $Y \gg e^{\frac{1}{\sg-1}}$, where the implicit constant is uniform in $Y$. 
\end{lem}
\begin{rem} \label{REM1}
We remark that the first and last sets of estimates are best possible.  Indeed, if we let $A > 1$ and set $E_j := \{p \in \mc{P} : p > e^{\frac{A}{\sg-1}}\}$ then it follows from Lemma \ref{SIMPLE} that $\beta_j \sim (1+A)e^{-A}$ while $\alpha_j \sim e^{-A}$, for any fixed $\sg > 1$. The first chain of estimates is thus best possible (since the first estimate is morally the same as the second one).  Also, if $E_j := \{p \in \mc{P} : p \leq \log\left(\frac{1}{\sg-1}\right)\}$ then $\alpha_j \sim (\sg-1)\log_2\left(\frac{1}{\sg-1}\right)$, while $\beta_j \sim (\sg-1)^2\log_2^2\left(\frac{1}{\sg-1}\right)$, so the last assertion is also best possible in this sense.
\end{rem}
\begin{proof}
We first prove $\beta_j \ll \alpha_j \log\left(\frac{1}{\alpha_j}\right)$. Observe that for $M \in \N$ fixed,
\begin{align*}
\beta_j &\leq (\sg-1)M \sum_{p \in E_j \atop p \leq e^{\frac{M}{\sg-1}}} \frac{\log p}{p^{\sg}} + (\sg-1)^2\sum_{p \in E_j \atop p > e^{\frac{M}{\sg-1}}} \frac{\log^2 p}{p^{\sg}}\\
&\leq M\alpha_j + (\sg-1)^2\sum_{p \in E_j \atop p > e^{\frac{M}{\sg-1}}} \frac{\log^2 p}{p^{\sg}}.
\end{align*}
By Lemma \ref{SIMPLE},
\begin{equation*}
\sum_{p > e^{\frac{M}{\sg-1}}} \frac{\log^2 p}{p^{\sg}} \ll (\sg-1)^{-2}e^{-M},
\end{equation*}
so that $\beta_j \leq M\alpha_j + e^{-M}$. Choosing $M := \llf \log\left(\frac{1}{\alpha_j}\right)\rrf$ gives $\beta_j \ll \alpha_j \log\left(\frac{1}{\alpha_j}\right)$. The proof that $\gamma_j \ll \beta_j \log\left(\frac{1}{\beta_j}\right)$ is similar, changing $(\sg-1)^j$ to $(\sg-1)^{j+1}$, $\log^j$ to $\log^{j+1}$, for $j \in \{1,2\}$. To prove $\gamma_j \ll \alpha_j \log^2\left(\frac{1}{\alpha_j}\right)$ requires the last estimate, so we shall prove this first.\\
Let $Y^+ := (\sg-1)\log Y$. Define 
\begin{equation*}
N_j(Y) := \inf\left\{ 0 < t < Y^+ : (\sg-1)\sum_{p \in E_j \atop p \leq e^{\frac{t}{\sg-1}}} \frac{\log p}{p^{\sg}} \geq \frac{1}{2}\alpha_j(Y)\right\}.  
\end{equation*}
Then we have
\begin{equation}\label{AM}
\alpha_j(Y) \leq 2(\sg-1)\sum_{p \in E_j \atop p \leq e^{\frac{N_j(Y)}{\sg-1}}} \frac{\log p}{p^{\sg}} \leq 2(\sg-1)\sum_{p \leq e^{\frac{N_j(Y)}{\sg-1}}} \frac{\log p}{p^{\sg}} \leq 2(\sg-1)\sum_{p \leq e^{\frac{N_j(Y)}{\sg-1}}} \frac{\log p}{p} \leq 3N_j(Y),
\end{equation}
this last estimate following by Mertens' (first) theorem, which is uniform in $Y$ (see I.2 in \cite{Ten2}). On the other hand, we have
\begin{align*}
\alpha_j(Y) &\leq (\sg-1)\sum_{p \in E_j \atop p \leq e^{\frac{N_j(Y)}{2(\sg-1)}}} \frac{\log p}{p^{\sg}} + (\sg-1)\sum_{p \in E_j \atop  e^{\frac{N_j(Y)}{2(\sg-1)}} < p \leq Y} \frac{\log p}{p^{\sg}} \\
&\leq \frac{1}{2}\alpha_j(Y) + \frac{2}{N_j(Y)}(\sg-1)^2\sum_{p \in E_j \atop p \leq Y} \frac{\log^2 p}{p^{\sg}} = \frac{1}{2}\alpha_j(Y) + \frac{2}{N_j(Y)}\beta_j(Y)
\end{align*}
It follows that $\alpha_j(Y) \leq \frac{4}{N_j(Y)}\beta_j(Y) \leq \frac{4}{N_j(Y)}\beta_j(Y)$. Combined with \eqref{AM}, this gives $\alpha_j^2(Y) \leq 12 \beta_j(Y)$, implying the second estimate. \\
From this last estimate with $Y \ra \infty$ we see that 
\begin{equation*}
\gamma_j \ll \beta_j \log\left(\frac{1}{\beta_j}\right) \ll \alpha_j \log\left(\frac{1}{\alpha_j}\right)\log\left(\frac{1}{\beta_j}\right) \ll \alpha_j \log^2 \left(\frac{1}{\alpha_j}\right).
\end{equation*}
Lastly, as $te^{-\frac{1}{2}t}$ is decreasing for $t > 2$, it follows that
\begin{align*}
(\sg-1)^3\sum_{p \in E_j \atop p > e^{\frac{2}{\sg-1}}} \frac{\log^3 p}{p^{\sg}} &= (\sg-1)^2\sum_{p \in E_j \atop p > e^{\frac{2}{\sg-1}}} \frac{\log^2 p}{p^{\sg'}} \left((\sg-1)\log pe^{-\frac{1}{2}(\sg-1)\log p}\right) \\
&\leq 8e^{-1}(\sg'-1)^2\sum_{p \in E_j \atop p > e^{\frac{2}{\sg-1}}} \frac{\log^2 p}{p^{\sg'}} \leq 8e^{-1}\beta_{j,\sg'},
\end{align*}
the last expression following because $\sg' - 1 = \frac{1}{2}(\sg-1)$ and $(\sg-1)\log p e^{-\frac{1}{2}(\sg-1)\log p} \leq 2e^{-1}$ for $p > e^{\frac{2}{\sg-1}}$.  As before, since $(\sg-1)^3\sum_{p \in E_j \atop p \leq e^{\frac{2}{\sg-1}}} \frac{\log^3 p}{p^{\sg}} \leq 2\beta_{j,\sg}$, the second claim follows as well. \\
Assuming hypothesis $\mathbf{H}_3(\sg)$, we have $\sum_{p \in E_{j''} \atop p > e^{\frac{2}{\sg-1}}} \frac{\log^2 p}{p^{\sg}} \ll \sum_{p \in E_{j''}} \frac{\log^2 p}{p^{\sg}}$, with the notation there. Since $(\sg'-1) \asymp (\sg-1)$ and, by the above proof, $\gamma_{j'',\sg} \ll \beta_{j'',\sg} + (\sg-1)^2\sum_{p \in E_{j''} \atop p > e^{\frac{2}{\sg-1}}} \frac{\log^2 p}{p^{\sg}}$, this last term is $\ll \beta_{j'',\sg}$ as well, which shows that $\gamma_{j'',\sg} \ll \beta_{j'',\sg}$. This is precisely $\mathbf{H}_2(\sg)$, by the definitions of $\beta_j$ and $\gamma_j$.
\end{proof} 
%
%
Furthermore, we will need to the following estimate related to the angular distribution of primes, i.e., the distribution of the values $\{\frac{t \log p}{2\pi}\}$ in $[0,1]$, where $t \in \mb{R}$ is a parameter. 
\begin{lem} \label{ANGLES}
Let $y \geq 2$. Let $t \in \R$ be such that $|t|\log y > 1$. Let $0 \leq \alpha < \beta \leq 1$, and define $\mc{I} := [\alpha,\beta] \subset [0,1]$.  Further, define $\theta_p(t) := \left\{\frac{t \log p}{2\pi}\right\}$, for each prime $p$. Then uniformly in $y$,
\begin{equation}
\sum_{p \leq y \atop \theta_p(t) \in \mc{I}} \frac{1}{p} = \left(1+O\left(\frac{1}{\log y} + \frac{1}{\log(|t|\log y)} \left(\frac{1}{|t|\log y} + 1\right)\right)\right)(\beta-\alpha) \log\left(|t|\log y\right) . \label{INTERVAL}
\end{equation}
\end{lem}
\begin{rem}\label{REM2}
We remark that, according to the proof to follow, the implicit constant in front of the term $\frac{1}{\log\left(|t|\log y\right)}$ is at most 7.  This will be important in the context of Lemma \ref{LemSmallTrunc}.
\end{rem}
\begin{proof}
Let $\chi_{\mc{I}}$ denote the characteristic function of the interval $\mc{I}$. By partial summation,
\begin{align*}
\sum_{p \leq y \atop \theta_p(t) \in \mc{I}} \frac{1}{p} &= \left(1+O\left(\frac{1}{\log y}\right)\right)\int_{2}^y \chi_{\mc{I}}\left(\left\{\frac{t \log u}{2\pi}\right\}\right) \frac{du}{u\log u} \\
&= \left(1+O\left(\frac{1}{\log y}\right)\right)\int_{\log 2}^{\log y} \chi_{\mc{I}}\left(\left\{\frac{tu}{2\pi}\right\}\right) \frac{du}{u}.
\end{align*}
Define $M := \llf \frac{t\log y}{2\pi} \rrf$. We make the change of variables $u \mapsto \frac{tu}{2\pi}$ and restrict the integral to $[1,M]$, which introduces an error of size at most $\frac{1}{|t|\log y}$. Decomposing this integral as a sum of integrals of unit length gives
\begin{align*}
\sum_{p \leq y \atop \theta_p(t) \in \mc{I}} \frac{1}{p} &= \left(1+O\left(\frac{1}{\log y}\right)\right)\sum_{1 \leq k \leq M} \int_{k+\alpha}^{k+\beta} \frac{du}{u} + O(1) = \left(1+O\left(\frac{1}{\log y}\right)\right)\sum_{1 \leq k \leq M} \int_{\alpha}^{\beta} \frac{du}{u+k} \\
&= \left(1+O\left(\frac{1}{\log y}\right)\right)\sum_{1 \leq k \leq M} \left(\frac{1}{k} \int_{\alpha}^{\beta} du - \int_{\alpha}^{\beta}\frac{u du}{k(u+k)}\right) .
\end{align*}
It is well-known (see, for instance, I.2 in \cite{Ten2}) that $\left|\sum_{1 \leq k \leq m} \frac{1}{k} - \log m - \gamma\right| \leq \frac{1}{m}$, where $\gamma$ is the Euler-Mascheroni constant. As $\int_{\alpha}^{\beta} \frac{u}{k(u+k)} du \leq \frac{1}{2k^2}(\beta^2-\alpha^2)$, we have
\begin{align*}
\sum_{p \leq y \atop \theta_p(t) \in \mc{I}} \frac{1}{p}&= \left(1+O\left(\frac{1}{\log y}\right)\right)(\beta-\alpha) \left(\log M  + \gamma +O\left(\frac{1}{M} + (\beta+\alpha)\right)\right) \\
&= \left(1+O\left(\frac{1}{\log y} + \frac{1}{\log(|t|\log y)} \left(\frac{1}{|t|\log y} + \beta+\alpha\right)\right)\right)(\beta-\alpha)\log M.
\end{align*}
As $\beta+\alpha \leq 2$, and $|\log M - \log(|t|\log y)| \leq 2\log(2\pi)$, the claim follows with an error at most $\leq \gamma + 2 + 2\log(2\pi) \leq 7$ in front of $\frac{1}{|t|\log Y}$.
\end{proof}
Lastly, we will need, in a couple of places in Lemma \ref{LemAppTrunc}, the following simple integral estimate.
\begin{lem}\label{GenIntEst}
Let $\gamma \geq 2$ and suppose $u > 0$.  Then
\begin{equation*}
\int_u^{\infty} dt(1+t^2)^{-\gamma} \leq \sqrt{\frac{\pi^2}{2\gamma}}\left(1+u^2\right)^{-\frac{\gamma}{2}}.
\end{equation*}
\end{lem}
\begin{proof}
Observe that $(1+\tau^2)^{\frac{\gamma}{2}} \geq 1+\frac{\gamma}{2}\tau^2$ for $\tau > 0$, implying that $(1+\tau^2)^{\gamma} \geq (1+\frac{\gamma}{2}\tau^2)(1+\tau^2)^{\frac{\gamma}{2}}$.  Estimating trivially, $\left(1+\tau^2\right)^{\frac{\gamma}{2}} \geq (1+u^2)^{\frac{\gamma}{2}}$, for $\tau \geq u$, hence
\begin{equation*}
\int_u^{\infty} \frac{d\tau}{(1+\tau^2)^{\gamma}} \leq (1+u^2)^{-\frac{\gamma}{2}} \int_{u}^{\infty} \frac{d\tau}{1+\frac{\gamma}{2}\tau^2} \leq \sqrt{\frac{2}{\gamma}}(1+u^2)^{-\frac{\gamma}{2}}  \int_0^{\infty} \frac{d\tau}{1+\tau^2} = \sqrt{\frac{\pi^2}{2\gamma}}(1+u^2)^{-\frac{\gamma}{2}}.
\end{equation*}
\end{proof}
\section{Existence and Uniqueness of a Saddle Point}
Our first important objective, as outlined in Section 2, is to prove the existence and uniqueness of the parameter vector $(\mbf{\rho},\sg)$. 
\begin{lem} \label{LemEU}
Assume that each $k_j \geq 1$. Then there exists a unique solution vector $(\mbf{\rho},\sigma)$ such that the equations $f_{z_j}(\mathbf{\rho},\sigma) = \frac{k_j}{\rho_j}$ and $f_{s}(\mathbf{\rho},\sigma) = -\log x$ are simultaneously satisfied.
\end{lem}
\begin{proof}
Let $y := \exp\left(\frac{\log_2 x}{\xi_1(x)}\right)$ and $z := \text{exp}\left(\xi_2(x)\log x \log_2 x\right)$ for some functions $\xi_1$ and $\xi_2$ satisfying $\xi_j(x) \ra \infty$ as $x \ra \infty$. Set $\sg_1 = 1+\frac{1}{\log y}$ and $\sg_2 := 1+\frac{1}{\log z}$. Without loss of generality, assume that $n$ is such that $\sum_{p \in E_j} \frac{\log p}{p^{\sg_2}} \geq \frac{1}{n+1} \frac{1}{\sg_2-1}$, by Lemma \ref{LAURENTZETA}. We begin by observing that by fixing $\sigma > 1$ we can choose real-valued functions $(\rho_0(\sigma),\ldots,\rho_{n-1}(\sigma),\rho_n(\sigma))$ such that
\begin{equation}
\sum_{0 \leq j \leq n} \rho_j(\sigma) \sum_{p \in E_j} \frac{p^{\sigma}\log p}{(p^{\sigma}-1)(p^{\sigma} - 1 + \rho_j(\sigma))} = \log x \label{LOGS},
\end{equation}
and, simultaneously, for each $0 \leq j \leq n-1$,
\begin{align}
\rho_j(\sigma)\sum_{p \in E_j} \frac{1}{p^{\sigma} -1 + \rho_j(\sigma)} &= k_j \label{NONLOGS}.
\end{align}
That a unique value of $\rho_j(\sg)$ exists for each $\sg$ follows because the left hand side of \eqref{NONLOGS} is monotone in $\rho_j$, by straightforward differentiation. The choices of $\rho_j(\sigma)$ depend continuously on $\sigma$ in each of the equations parametrized by $0 \leq j \leq n-1$, and $\rho_n(\sigma)$ may be deduced from the remaining ones as a continuous function as well.  
We note that \eqref{LOGS} can be rewritten as
\begin{align}
\sum_{0 \leq j \leq n} \rho_j(\sg) \sum_{p \in E_j} \frac{\log p}{p^{\sg}-1+\rho_j} &= \left(1+O\left(\frac{1}{\log x}\sum_{0 \leq j \leq n} \rho_j\right)\right) \log x \nonumber\\
&= \left(1+O\left(\frac{1}{\log_2 x}\right)\right)\log x \label{LOGS2},
\end{align}
since, by the assumption that $E_j(t) \ra \infty$ as $t \ra \infty$, \eqref{NONLOGS} implies that $\rho_j \ll k_j \ll \frac{\log x}{\log_2 x}$ when $\sg$ is sufficiently small. We wish to check that for some choice of $\sigma$, the additional constraint $k_n = \rho_n(\sigma) \sum_{p \in E_n} \frac{1}{p^{\sigma} - 1 + \rho_n(\sigma)}$ also holds, and to this end it suffices to show that for $\sg = \sg_2$ the sum in question is rather small, while for $\sg = \sg_1$, it exceeds $k_n$. By continuity, some solution $\sg \in (\sg_2,\sg_1]$ must exist.\\
To see that it is small, assume that $\sigma_2 := 1+\frac{1}{\log z}$. According to our choice of labelling,
\begin{equation*}
\sum_{p \in E_n} \frac{\log p}{p^{\sigma_2}} \geq \frac{1}{2(n+1)} \frac{1}{\sg_2-1} = \frac{1}{2(n+1)}\log z.
\end{equation*}
Note that $\rho(\sg_2) = o(\log z)$ because $\rho_n(\sg_2) \leq \log x$ and $\log x = o(\log z)$ by assumption. Thus, we have
\begin{align*}
\log x &\geq \frac{1}{2}\rho_n(\sigma_2) \sum_{p \in E_n} \frac{\log p}{p^{\sigma_2}-1+\rho_n(\sigma_2)} \\
&= \rho_n(\sigma_2)\sum_{p \in E_n} \frac{\log p}{p^{\sigma_2}} + \rho_n(\sigma_2)\sum_{p \in E_n \atop p \leq \rho_n(\sigma_2)^{1/\sigma_2}} \log p\left(\frac{1}{p^{\sigma_2}-1+\rho_n(\sigma_2)} - \frac{1}{p^{\sigma_2}}\right) \\
&= \rho_n(\sigma_2)\left(\sum_{p \in E_n} \frac{\log p}{p^{\sigma_2}} + O(\log \rho_n)\right) \geq \frac{\rho_n(\sigma_2)}{4(n+1)}(\log z),
\end{align*}
the second last estimate following from
\begin{equation*}
\rho_n(\sg)^2\sum_{p \in E_n \atop p \leq \rho_n^{\frac{1}{\sg}}} \frac{\log p}{p^{\sg}(p^{\sg}-1+\rho_j(\sg)} \ll \rho_n(\sg)\sum_{p \in E_n \atop p \leq \rho_n^{\frac{1}{\sg}}} \frac{\log p}{p} \ll \rho_n(\sg)\log\left(1+ \rho_n(\sg)\right),
\end{equation*}
 for $\sg > 1$, by Mertens' first theorem (\cite{Ten2}). It follows that $\rho_n(\sigma_2) \leq \frac{8(n+1)\log x}{\log z}$. In particular, when $x$ is sufficiently large in terms of $n$, we have
\begin{equation*}
\sum_{p \in E_n} \frac{\rho_n(\sigma_2)}{p^{\sigma_2}-1+\rho_n(\sigma_2)} \leq \frac{8(n+1)\log x}{\log z} \sum_{p \in E_n} \frac{1}{p^{\sigma_2}-1} \leq \frac{8(n+1)\log x \log_2 x}{\log z} \leq \frac{8(n+1)}{\xi_2(x)}.
\end{equation*}
Therefore, certainly,
\begin{equation*}
\sum_{p \in E_n} \frac{\rho_n(\sigma_2)}{p^{\sigma_2}-1+\rho_n(\sigma_2)} < 1 \leq k_n, 
\end{equation*}
for sufficiently large $x$, the last inequality being true by assumption.\\
Assume now that for each $\sigma \geq \sigma_2$ we have $\sum_{p \in E_n} \frac{\rho_n(\sigma) }{p^{\sigma}-1+\rho_n(\sigma)} < k_n$.  In particular, suppose $\sg = \sg_1$. It then follows that $\sum_{p \in E_n\atop p \leq y } \frac{\rho_n(\sigma_1) \log p}{p^{\sigma_1} - 1 + \rho_n(\sigma_1)} < k_n \log y$, and hence
\begin{align*}
\log x &\ll \sum_{0 \leq j \leq n} \rho_j(\sigma_1)\left(\sum_{p \in E_j} \frac{\log p}{p^{\sigma_1}-1+\rho_j(\sigma_1)} + \sum_{p \in E_j} \frac{\log p}{p^{\sigma_1} -1 +\rho_j(\sigma_1)}\right) \\
&< \left(\sum_{0 \leq j\leq n} k_j\right)\log y +  \sum_{0 \leq j \leq n} \rho_j(\sigma_1) \sum_{p \in E_j \atop p > y} \frac{\log p}{p^{\sigma_1} - 1 + \rho_j(\sigma_1)}.
\end{align*}
Since $\rho_j \ll k_j$, as mentioned above, we have
\begin{align*}
\sum_{0 \leq j \leq n} \rho_j(\sigma_1)\sum_{p \in E_j \atop p > y} \frac{\log p}{p^{\sigma_1}-1+\rho_j(\sigma_1)} &\ll \left(\sum_{0 \leq j \leq n} k_j\right)\sum_{0\leq j \leq n} \sum_{p \in E_j \atop p > y} \frac{\log p}{p^{\sigma_1} - 1} \asymp \left(\sum_{0 \leq j \leq n} k_j\right) \frac{y^{1-\sigma_1}}{\sigma_1 - 1} \\
&\asymp \left(\sum_{0 \leq j \leq n} k_j\right)\log y.
\end{align*}
where the second last expression comes from Lemma \ref{SIMPLE}. Suppose $y > \max_{0 \leq j \leq n} \rho_j \geq \max_{0 \leq j \leq n} \rho_j^{\frac{1}{\sg}}$ for $\sg > 1$. Since $\mbf{k}$ is a vector such that there exist an integer $m \leq x$ such that $k_0 + \ldots + k_n = \omega(m) \leq \frac{\log x}{\log_2 x}$, it follows that
\begin{equation*}
\log x \ll \left(\sum_{0 \leq j \leq n} k_j\right)\log y \ll \left(\frac{\log x}{\log_2 x}\right)\log y \ll \frac{\log x}{\xi_1(x)},
\end{equation*}
a contradiction for sufficiently large $x$. Thus, we must indeed have 
\begin{equation*}
\rho_n(\sigma_1) \sum_{p \in E_n} \frac{1}{p^{\sigma_1} - 1 + \rho_n(\sigma_1)} \geq k_n. 
\end{equation*}
Since $\rho_n$ is continuous, it follows that for some $\sigma \in (\sigma_2, \sigma_1]$, all of the equations are solvable, as claimed. \\
To show that $\sg$ is unique (whence $\rho_0,\ldots,\rho_n$ are uniquely determined) we consider the map 
\begin{equation*}
g(\mbf{\rho},\sg) := (\rho_0f_{\rho_0}(\mbf{\rho};\sg),\ldots,\rho_nf_{\rho_n}(\mbf{\rho};\sg),f_s(\mbf{\rho};\sg)), 
\end{equation*}
and claim that $g(\mbf{\rho},\sg)$ is invertible whenever $(\rho_0,\ldots,\rho_n) \in \mb{R}^{n+1}_+$ and $\sg \in (1,\infty)$. This shows that there is at most one solution in positive real parameters to $g(\mbf{\rho};\sg) = (k_0,\ldots, k_n,-\log x)$ which, coupled with the conclusion of the previous paragraphs, implies the uniqueness of our solution.  Let us remark, moreover, that since $\sg_1 \ra 1^+$ as $x \ra \infty$ and $1 < \sg \leq \sg_1$, so does $\sg$.\\
The Jacobian matrix of $g$ is
\begin{align*}
J_g(\mbf{\rho},\sg) &:= \left(\begin{matrix} f_{z_0}+\rho_0f_{z_0z_0} &\rho_0f_{z_1z_0} &\cdots &\rho_0f_{z_nz_0} &\rho_0f_{sz_0} \\ \vdots &\vdots &\ddots & \vdots & \vdots \\ \rho_n f_{z_0z_n} &\rho_nf_{z_1z_n} &\cdots &f_{z_n}+\rho_nf_{z_nz_n} &\rho_nf_{sz_n} \\
f_{sz_0} & f_{sz_1} &\cdots &f_{sz_n} &f_{ss}\end{matrix}\right) \\
\end{align*}
where we recall that for each $0 \leq j \leq n$,
\begin{align*}
f_{z_j}+\rho_jf_{z_jz_j} &= \sum_{p \in E_j} \frac{p^{\sg}-1}{(p^{\sg}-1+\rho_j)^2}, \\
f_{sz_j} &= -\sum_{p \in E_j} \frac{p^{\sg} \log p}{(p^{\sg}-1+\rho_j)^2}, \\
f_{ss} &= \sum_{0 \leq m \leq n} \rho_m \sum_{p \in E_m} \frac{p^{\sg} \log^2 p}{(p^{\sg}-1)(p^{\sg}-1+\rho_m)}\left(\frac{1}{p^{\sg}-1} + \frac{1}{p^{\sg}-1+\rho_m}\right). 
\end{align*}
By the inverse function theorem, $g(\mbf{\rho},\sg)$ is invertible locally at $(\mbf{\rho},\sg)$ if, and only if, $\det(J_g(\mbf{\rho},\sg)) \neq 0$.  For simplicity, we write $J_g(\mbf{\rho},\sg) = (a_{ij})_{0 \leq i,j \leq n+1}$, where $a_{ij} = 0$ whenever $0 \leq i \leq n$ and $j \neq i,n+1$. Letting $S_{n+1}$ denote the permutation group on $n+1$ elements, and denoting by $(a \ b)$ the permutation swapping $a$ and $b$ while leaving all other elements fixed, we have
\begin{align*}
\det(J_g(\mbf{\rho},\sg)) &= \sum_{\sigma \in S_{n+1}} (-1)^{\text{inv}(\sigma)}\prod_{0 \leq j \leq n+1} a_{j\sigma(j)} = \sum_{\sigma \in S_{n+1} : \sigma = (j \ n+1) \atop 0 \leq j \leq n+1} (-1)^{\text{inv}(\sigma)}\prod_{0 \leq i \leq n+1} a_{i\sigma(i)} \\
&= \prod_{0 \leq i \leq n+1} a_{ii} - \sum_{0 \leq j \leq n} a_{j,n+1}a_{n+1,j}\prod_{0 \leq i \leq n \atop i \neq j} a_{ii} \\
&= \prod_{0 \leq i \leq n} a_{ii}\left(a_{n+1,n+1}-\sum_{0 \leq j \leq n} \frac{a_{j,n+1}a_{n+1,j}}{a_{jj}}\right).
\end{align*}
Now, by definition, we have
\begin{align}
&a_{n+1,n+1} - \sum_{0 \leq j \leq n} \frac{a_{j,n+1}a_{n+1,j}}{a_{jj}} \label{COND2}\\
&= \sum_{0 \leq j \leq n} \rho_j\sum_{p \in E_j} \frac{p^{\sg} \log^2 p}{(p^{\sg}-1)(p^{\sg}-1+\rho_j)}\left(\frac{1}{p^{\sg}-1}+ \frac{1}{p^{\sg}-1+\rho_j}\right) \\
&- \sum_{0 \leq j \leq n} \rho_j\left(\sum_{q \in E_j} \frac{1}{(q^{\sg}-1+\rho_j)^{2}}\right)^{-1}\left(\sum_{p_1,p_2 \in E_j} \frac{(p_1p_2)^{\sg} (\log p_1)(\log p_2)}{(p_1^{\sg}-1+\rho_j)^2(p_2^{\sg}-1+z_j)^2}\right) \nonumber.
\end{align}
Consider $\rho_0, \ldots, \rho_n\in (0,\infty)$ and $\sg \in (1,\infty)$. It suffices to show that for each $0 \leq j \leq n$,
\begin{align}
0 &> 
 \sum_{p_1,p_2 \in E_j} \frac{p_1^{\sg} \log p_1}{(p_1^{\sg}-1 + \rho_j)(p_2^{\sg}-1 + \rho_j)^2} \left(\frac{\log p_1}{p_1^{\sg}-1}\left(\frac{1}{p^{\sg}-1} + \frac{1}{p^{\sg}-1+\rho_j}\right) - \frac{p_2^{\sg} \log p_2}{p_1^{\sg}-1 + \rho_j}\right). \label{S1}
\end{align}
This is clearly true for each $j$, as $a_{n+1,j}a_{j,n+1} \asymp \left(\sum_{p \in E_j} \frac{\log p}{p^{\sg}-1+\rho_j}\right)^2 \ra \infty$ as $x \ra \infty$ (since, as we showed above, $\sg \ra 1^+$ in this case), while $a_{n+1,n+1}$ and $a_{jj}$ are both bounded for all choices of $\sg,\rho_j$.
Thus, the expression \eqref{COND2} is negative for $x$ sufficiently large. Since each $a_{jj} = f_{\rho_j}+\rho_jf_{z_jz_j} > 0$ for this range of parameters, $\det(J_g(\mbf{\rho},\sg)) < 0$, which completes the proof.
\end{proof}
We henceforth write $\rho_j(\sg)$ as $\rho_j$ for convenience; as $\sg$ is fixed and $\mbf{\rho}$ is completely determined by it, this is justified. \\
As an artifact of the proof of Lemma \ref{LemEU}, we see that the unique solution $(\mbf{\rho},\sigma)$ occurs when $\frac{1}{\xi_1(x)\log x \log_2 x} \leq \sigma-1 \leq \frac{\xi_2(x)}{\log_2 x}$. In fact, much stronger upper and lower bounds are available. We may also find, as a result, effective estimates for $\rho_j$ (note that whenever $k_j \geq 1$, $\rho_j > 0$ by \eqref{NONLOGS}).
\begin{lem} \label{PARAMS}
Let $\e > 0$ and $x \geq 3$. Suppose that $1 \leq k_j \leq \log^{1-\e} x$ for each $0 \leq j \leq n$. Let $\mbf{z} \in \mb{C}^{n+1}$ such that $z_j \neq 0$ for each $j$, and write $\|\mbf{z}\| := \left(\sum_{0\leq j \leq n} |z_j|^2\right)^{\frac{1}{2}}$ and $1/\mbf{z} := (\frac{1}{z_0},\ldots,\frac{1}{z_n})$.
Then the vector $\mbf{\rho}$ in Lemma \ref{LemEU} satisfies the condition that
\begin{equation}
(\sg-1)\log x = \left(1+O\left(\sum_{0 \leq j \leq n} E_j(x)^{-1}\right)\right)\sum_{0 \leq j \leq n} \rho_j\alpha_j \label{EXACTSG},
\end{equation}
and in particular,
\begin{align}
(\sigma - 1)\log x &\leq \|\mbf{\rho}\|\left(1+O\left(\frac{\|\mbf{\rho}\|}{\log x}\right)\right) \label{UPPER}\\
(\sg - 1)\log x &\geq \|1/\mbf{\rho}\|^{-1} + O\left(\frac{\|\mbf{\rho}\|}{\log x}\right). \label{LOWER}
\end{align}
For each $0 \leq j \leq n$, $\rho_j$ satisfies
\begin{equation*}
k_j = g_j(\sg) - (\rho_j-1)\sum_{p \in E_j} \frac{1}{p^{\sg}(p^{\sg}-1+\rho_j)}.
\end{equation*}
In particular, if $E_j(x) \ra \infty$ as $x \ra \infty$ and $\sum_{0 \leq j \leq n} \rho_j \ll (\log x)^{1-\e}$ for $\e > 0$ fixed then
\begin{equation*}
\rho_j = \frac{k_j}{E_j\left(e^{\frac{1}{\sg-1}}\right)}\left(1+O_{\e}\left(\frac{E_j\left(k_j/E_j\left(e^{\log^{\e} x}\right)\right)}{E_j\left(e^{\log^{\e}x}\right)}\right)\right).
\end{equation*}
\end{lem}
\begin{rem}\label{REM3}
According to the second conclusion of the lemma, we have $(1+o(1))(\sg - 1)\log x \in [\|1/\mbf{\rho}\|^{-1},\|\mbf{\rho}\|]$.  Note that when $\rho_j \asymp \rho_{j'}$ for each $j,j'$, this gives $(\sg-1)\log x \asymp_n \|\rho\|$. Thus, \eqref{UPPER} and \eqref{LOWER} are sharp in those cases in which both the values of $E_j\left(e^{\frac{1}{\sg-1}}\right)$ are all similar, as is the case in Lemma \ref{LemDirDens}, and when the values of $k_j$ are all similar. \\
According to the last statement of the lemma, $\rho_j \leq k_j \ll \frac{\log x}{\log_2 x}$ uniformly in $j$, so the condition $\sg - 1 \ll \frac{1}{\log_2 x}$ is indeed fulfilled, as discussed at the end of Section 2.  
\end{rem}
\begin{proof} 
For the first claim, we simply observe that (cf. \eqref{LOGS2})
\begin{align*}
\log x &= \sum_{0 \leq j \leq n} \rho_j \sum_{p \in E_j} \frac{p^{\sg} \log p}{(p^{\sg}-1)(p^{\sg}-1+\rho_j)} \\
&= \sum_{0 \leq j \leq n} \rho_j \left(\sum_{p \in E_j} \frac{\log p}{p^{\sg}} - \sum_{p \in E_j} \left(\frac{\log p}{(p^{\sg}-1)(p^{\sg}-1+\rho_j)} + (\rho_j-1)\frac{\log p}{p^{\sg}(p^{\sg}-1+\rho_j)}\right)\right).
\end{align*}
Observe that
\begin{align*}
\rho_j^2 \sum_{p \in E_j \atop p \leq \rho_j^{\frac{1}{\sg}}} \frac{\log p}{p^{\sg}(p^{\sg}-1+\rho_j)} \ll \rho_j \sum_{p \leq \rho_j^{\frac{1}{\sg}}} \frac{\log p}{p} \ll \rho_j \log \rho_j\\
\rho_j^2 \sum_{p \in E_j \atop p > \rho_j^{\frac{1}{\sg}}} \frac{\log p}{p^{\sg}(p^{\sg}-1+\rho_j)} \ll \rho_j \log \rho_j,
\end{align*}
the second estimate coming from Lemma \ref{SIMPLE}. Hence, since $\rho_j \ll \frac{k_j}{E_j\left(e^{\frac{1}{\sg-1}}\right)}$ from the second last claim of the lemma (which is proved independently of this one), and $k_j \ll \log^{1-\e} x$, the above sums contribute $O\left(\frac{\log^{1-\e} x}{E_j\left(e^{\frac{1}{\sg-1}}\right)}\right)$ to each one. As
\begin{equation*}
\sum_{p \in E_j} \frac{\log p}{p^{\sg}} = \alpha_j\left(\frac{1}{\sg-1}+O(1)\right),
\end{equation*}
and $\log^{-\e} x \ll \log_2^{-1} x \ll E_j(x)^{-1}$, we have
\begin{align*}
(\sg-1)\log x &= \left(1+O\left(\sum_{0 \leq j \leq n} E_j(x)^{-1}\right)\right) \sum_{0 \leq j \leq n} \rho_j\alpha_j\left(1+O\left(\frac{\left(\sum_{0 \leq j \leq n} \rho_j\right)^2}{\log^2 x}\right)\right) \\
&= \left(1+O\left(\sum_{0 \leq j \leq n} E_j(x)^{-1}\right)\right) \sum_{0 \leq j \leq n} \rho_j\alpha_j.
\end{align*}
This establishes the first claim. \\
For the second, note that the upper bound comes from Cauchy-Schwarz via
\begin{equation*}
\sum_{0 \leq j \leq n} \rho_j\alpha_j \leq \|\mbf{\rho}\|\sum_{0 \leq j \leq n} \alpha_j = \|\mbf{\rho}\|\left(1+O\left(\sg-1\right)\right),
\end{equation*}
so that $(\sg-1)\log x \leq \|\mbf{\rho}\|\left(1+O\left(\frac{\|\mbf{\rho}\|}{\log x}\right)\right)$.  For the lower bound, Cauchy-Schwarz also gives
\begin{equation*}
1 +O\left(\sg-1\right) = \sum_{0 \leq j \leq n} \alpha_j \leq \left(\sum_{0 \leq j \leq n} \rho_j\alpha_j\right)\|1/\mbf{\rho}\|, 
\end{equation*}
so that, as claimed,
\begin{equation*}
(\sg-1)\log x \geq \|1/\mbf{\rho}\|^{-1}\left(1+O\left(\sg-1\right)\right) = \|1/\mbf{\rho}\|\left(1+O\left(\frac{\|\rho\|}{\log x}\right)\right).
\end{equation*}
The first claim about $\rho_j$ follows from \eqref{Zit}, and $f_{z_j}(\mbf{\rho};\sg) = \frac{k_j}{\rho_j}$. Now, assume that $E_j(x) \ra \infty$ as $x \ra \infty$. To derive an estimate for $\rho_j$, we note that
\begin{align*}
\left|(\rho_j-1)\sum_{p \in E_j} \frac{1}{p^{\sg}(p^{\sg}-1+\rho_j)}\right| &\leq \sum_{p \in E_j \atop p \leq \rho_j^{1/\sg}} \frac{1}{p^{\sg}} + \rho_j\sum_{p \in E_i \atop p > \rho_i} \frac{1}{p^{2\sg}} \leq \sum_{ p \in E_j \atop p \leq \rho_j} \frac{1}{p} + O(1) = E_j(\rho_j) + O(1).
\end{align*}
Hence,
\begin{equation*}
\rho_j(\sg) = \frac{k_j}{E_j\left(e^{\frac{1}{\sg-1}}\right)}\left(1+O\left(\frac{E_j(\rho_j)}{E_j\left(e^{\frac{1}{\sg-1}}\right)}\right)\right), 
\end{equation*}
and since $k_j \ll (\log x)^{1-\e}$, also $\sg-1 \ll (\log x)^{\e}$, which proves the claim.
\end{proof}
The following result gives precise asymptotic expressions for $\sg$ and $\mbf{\rho}$ in the important and natural cases described by Theorem \ref{THM2}.
\begin{lem}\label{LemDirDens}
Suppose that $q \geq 2$ and each $E_j$ is a union of $m_j$ arithmetic progressions modulo $q$.  Assume, moreover, that $1 \leq k_j \ll \log^{1-\e} x$. Then if $K := \sum_{0 \leq j \leq n} k_j$,
\begin{align*}
(\sg-1)\log x &= \frac{K}{\log(\log x/K)}\left(1+O_{\e}\left(\frac{1}{\log(\log x/K)}\left(\phi(q)\sum_{0 \leq j \leq n} \frac{1}{m_j} + \log_3 x\right) \right)\right) \\
\rho_j &= \frac{\phi(q) k_j}{m_j \log\left(\frac{\log x}{K}\log(\log x/K)\right)}\left(1+O_{\e}\left(\frac{\log(\phi(q)k_j/m_j\log(\log x))}{\log_2 x}\right)\right)
\end{align*}
\end{lem}
\begin{proof}
It follows from the proof of Lemma \ref{LAURENTZETA} and from the previous lemma that
\begin{align*}
\rho_j &= k_j\left(\sum_{p \in E_j} \frac{1}{p^{\sg}}\right)^{-1}\left(1+O\left(\frac{\log_2(k_j/\log_2 x)}{\log(1/(\sg-1))}\right)\right) \\
&= \frac{\phi(q)k_j}{m_j}\log\left(\frac{1}{\sg-1}\right)^{-1}\left(1+O\left(\frac{1}{\log\left(\frac{1}{\sg-1}\right)}\left(\frac{\phi(q)}{m_j} + \log_3 x\right)\right)\right).
\end{align*}
On the other hand, Lemma \ref{LAURENTZETA} also states that 
\begin{equation*}
\sum_{p \in E_j} \frac{\log p}{p^{\sg}} = \frac{m_j}{\phi(q)}\frac{1}{\sg-1}\left(1+O\left(\frac{\phi(q)}{m_j}(\sg-1)\right)\right),
\end{equation*}
which, combined with $\left(1+O\left(\frac{1}{\log_2 x}\right)\right)\log x = \sum_{0 \leq j \leq n} \rho_j\sum_{p \in E_j} \frac{\log p}{p^{\sg}}$ from \eqref{LOGS2} leads to
\begin{align*}
\sg-1 &= \frac{1}{\log x}\frac{1}{\log\left(\frac{1}{\sg-1}\right)}\left(\sum_{0 \leq j \leq n} k_j\right)\left(1+O\left(\frac{1}{\log\left(\frac{1}{\sg-1}\right)}\left(\frac{\phi(q)}{m_j} + \log_3 x\right)\right)\right) \\
&= \frac{K}{\log x}\cdot \frac{1}{\log(\log x/K)}\left(1+O_{\e}\left(\frac{1}{\log(\log x/K)}\right)\left(\frac{\phi(q)}{m_j} + \log_3 x\right)\right),
\end{align*}
the last estimate following because $\log\left(\frac{1}{\sg-1}\right) \geq \e\log_2 x$.
\end{proof}
\section{Truncating the Integral in \eqref{PerrId}}
Having established precise estimates for the parameters $(\mbf{\rho},\sg)$ that enter into \eqref{PerrId}, we next set out to identify to what extent the remainder of the integral outside of a neighbourhood of $(\mbf{\rho},\sg)$ contributes to $\pi(x;\mbf{E},\mbf{k})$.  \\
The following estimate is a step towards establishing estimates for the ratio $\left|\frac{F(\mbf{z};s)}{F(\mbf{\rho};\sg)}\right|$ when $(\mbf{z},s)$ lies outside of some small neighbourhood of $(\mbf{\rho},\sg)$. This result is essentially contained in \cite{HiT}, but we repeat the argument, suitably adapted to our circumstances, since it is short.
\begin{lem} \label{TruEst1}
For any $\mbf{z} = (z_0,\ldots,z_n)$ and $s = \sg + i\tau$ then 
\begin{equation*}
|F(\mbf{z};s)| \ll F(\mbf{\rho};\sg)\exp\left(-\frac{1}{4}\sum_{0 \leq j \leq n} \rho_j \sum_{p \in E_j \atop p > \rho_j^{\frac{1}{\sg}}}\frac{1-\cos(\tau \log p -t_j)}{p^{\sg}}\right).
\end{equation*}
\end{lem}
\begin{proof}
Fix $\mbf{z} \in \mb{C}^{n+1}$, with $z_j = \rho_je^{it_j}$ and $t_j \in [-\pi,\pi]$. Also, define $\xi_j(p) := t_j-\tau \log p$. For each $0 \leq j \leq n$ and $p \in E_j$, we have
\begin{align*}
\left|1+\frac{z_j}{p^{s}-1}\right|^2 &\leq 1 + 2\text{Re}\left(\frac{z_j}{p^{s}}\right) + \frac{\rho_j^2}{|p^s-1|^2} + 2\frac{\rho_j}{p^{\sg}|p^s-1|} \\
&= 1 + 2\frac{\text{Re}(\rho_je^{i(t_j-\tau \log p)})}{p^{\sg}} + \frac{\rho_j^2}{p^{2\sg} + 1 - 2\rho_j \cos(\tau \log p)} + 2\frac{\rho_j}{p^{\sg}(p^{\sg}-1)}\\
&\leq 1 - 2\frac{\rho_j(1-\cos\xi_j(p))}{p^{\sg}} + 2\frac{\rho_j}{p^{\sg}}\left(1+\frac{1}{p^{\sg}-1}\right)+\frac{\rho_j^2}{(p^{\sg}-1)^2}  \\
&= 1-2\frac{\rho_j(1-\cos \xi_j(p))}{p^{\sg}} + 2\frac{\rho_j}{p^{\sg}-1}+\frac{\rho_j^2}{(p^{\sg}-1)^2}\\
&= \left(1+\frac{\rho_j}{p^{\sg}-1}\right)^2\left(1-2\frac{\rho_j(1-\cos\xi_j(p))}{p^{\sg}\left(1+\frac{\rho_j}{p^{\sg}-1}\right)^2}\right).
\end{align*}
By the above computation, it follows that 
\begin{align*}
|F(\mbf{z};s)| &= \prod_{0 \leq j \leq n} \prod_{p \in E_j} \left|1+\frac{\rho_je^{it_j}}{p^{\sg + i\tau}-1}\right| \leq \prod_{0 \leq j \leq n} \prod_{p \in E_j} \left(1+\frac{\rho_j}{p^{\sg}-1}\right)\left(1-2\frac{\rho_j(1-\cos\xi_j(p))}{p^{\sg}\left(1+\frac{\rho_j}{p^{\sg}-1}\right)^2}\right)^{\frac{1}{2}}.
\end{align*}
Restricting the ranges of summation to $p \in E_j$ and $p > \rho_j^{\frac{1}{\sg}}$, $\left(1+\frac{\rho_j}{p^{\sg}-1}\right)^2 \leq \frac{1}{4}$, and, moreover, from the elementary inequality $(1-x)^{\frac{1}{2}} \leq e^{-x/2}$ for $x \in (0,1)$,  for each $j$ and $p \in E_j$, $\left(1+\frac{\rho_j}{p^{\sg}-1}\right) \leq 2$, we get 
\begin{equation*}
|F(\mbf{z};s)| \leq F(\mbf{\rho};\sg)\exp\left(-\frac{1}{4}\sum_{0 \leq j \leq n}\rho_j\sum_{p\in E_j} \frac{1-\cos(t_j-\tau \log p)}{p^{\sg}}\right),
\end{equation*}
which yields the claim.
\end{proof}
In relation to Lemma \ref{TruEst1}, we define $M_j(\tau,t_j) := \sum_{p \in E_j \atop p > \rho_j^{\frac{1}{\sg}}} \frac{1-\cos(\tau \log p -t_j)}{p^{\sg}}$ for each $j$. The following simple lemma will be useful in the proofs of Lemma \ref{LemSmallTrunc}.
\begin{lem} \label{TrigProp}
Let $m \in \N$. Then for $a_1,\ldots,a_m \in \R$,
\begin{align}
\sin^2\left(\sum_{1 \leq j \leq m} a_j\right) \leq m\sum_{1 \leq j \leq m} \sin^2 a_j \label{PROP1}; \\
\sin^2\left(\frac{1}{m}\sum_{1 \leq j \leq m} a_j\right) \geq \left(\prod_{1 \leq j \leq m} \sin^2 a_j \right)^{\frac{1}{m}} \label{PROP2},
\end{align}
the second equation holding in the case when $\frac{1}{\pi m}\sum_{1 \leq j \leq m} a_j \notin \mb{Z}$. Hence, for $\tau,t,\tau_1,\tau_2, t_1,t_2 \in \R$,
\begin{align*}
M_j(\tau_1, t_1) + M_j(\tau_2,t_2) &\geq \frac{1}{2}M_j(\tau_1+\tau_2,t_1+t_2) \\
M_j(\tau,t) &\geq M_j(2\tau,0)\sin^2 t.
\end{align*}
\end{lem}
\begin{rem} \label{REM4}
Remark that the condition on \eqref{PROP2} implies that, for each integer $l$, the vector $(a_1,\ldots,a_m)$ lies on a subspace of $\mb{R}^m$ with positive codimension; it is thus a null set with respect to Lebesgue measure on $\mb{R}^m$ for each $l$, and hence the countable union of each of these subspaces is also null. We will therefore not concern ourselves with whether or not this condition is satisfied in any of our forthcoming applications of \eqref{PROP2}.
\end{rem}
\begin{proof}
For the first property, let $A_m := \sum_{1 \leq j \leq m} a_j$ and $A_{m-1} := A_m - a_m$. Then
\begin{align*}
&\sin^2 A_m = \left|\sin\left(A_{m-1} +a_m\right)\right|^2 = \left|\sin\left(A_{m-1}\right) \cos a_m + \sin a_m \cos \left(A_{m-1}\right)\right|^2 \\
&\leq \left(\left|\sin \left(A_{m-1}\right)\right| + \left|\sin a_m\right|\right)^2 \leq \left(\sum_{1 \leq j \leq m} \left|\sin a_m\right|\right)^2 \leq m\sum_{0 \leq j \leq n} \sin^2 a_j,
\end{align*}
the second last inequality following by induction from the previous line. For the second one, we simply note that $\frac{d^2}{dt^2} \log(\sin^2(t)) = 2\frac{d}{dt}\frac{\cos t}{\sin t} = -\frac{2}{\sin^2 t}$, whenever $t \notin \mb{Z}\pi$, which implies concavity on the domain of $\log(\sin^2 t)$, as claimed. \\
The last statement applies each of the first and second claim in sequence with $m = 2$, using the elementary identity $1-\cos u = 2\sin^2(u/2)$, and $|\sin u| \geq \sin^2 u$ for any $u \in \mb{\R}$.
\end{proof}
We seek to find a lower bound for $M_j(\tau,t_j)$ in order to get a bound on $\left|\frac{F(\mbf{z};s)}{F(\rho;\sg)}\right|$ that is suitable to truncate \eqref{PerrId} with a sufficiently small error term (which we accomplish in Lemma \ref{LemAppTrunc}).  To this end, we give estimates depending on the value of the arguments of $t_j$ and $\tau$. 
First, as a corollary of Lemma \ref{ANGLES}, we may derive the following estimate, which will factor especially in the proof of Lemma \ref{LemAppTrunc}.
\begin{lem} \label{EXTRADECAY}
Suppose that $1 < \frac{|\tau|}{\sg-1} \ll e^{E_j\left(e^{\frac{1}{\sg-1}}\right)}$. Then for each $j$ we have
\begin{equation*}
\sum_{0 \leq j \leq n} \rho_jM_j(\tau,0) \geq \frac{1}{24}\sum_{0 \leq j \leq n} k_j \min\left\{1,\frac{E_j\left(e^{\frac{1}{\sg-1}}\right)}{6\log\left(\frac{|\tau|}{\sg-1}\right)}\right\}^2.
\end{equation*}
\end{lem}
\begin{proof}
Let $\mu := \min\left\{\frac{1}{2},\frac{E_j\left(e^{\frac{1}{\sg-1}}\right)}{c\log\left(\frac{|\tau|}{\sg-1}\right)}\right\}$, where $c\geq 2$ is a sufficiently large constant (whose value will be chosen momentarily), and set $\mc{I} := (-\mu,\mu)$. By Lemma \ref{ANGLES} (applied in the interval $[-\frac{1}{2},\frac{1}{2}]$ rather than $[0,1]$, or equivalently, on the two intervals $[0,\frac{1}{2}]$ and $[-\frac{1}{2},0]$ separately),  
\begin{align*}
\sum_{p \in E_j \atop \theta_p(\tau) \in \mc{I}} \frac{1}{p^{\sg}} &= \sum_{p \in E_j \atop \theta_p(\tau) \in \mc{I}, p \leq e^{\frac{1}{\sg-1}}} \frac{1}{p} + O(1) \leq c\frac{E_j\left(e^{\frac{1}{\sg-1}}\right)}{2c\log \left(\frac{|\tau|}{\sg-1}\right)} \log\left(\frac{|\tau|}{\sg-1}\right) = \frac{1}{2}E_j\left(e^{\frac{1}{\sg-1}}\right).
\end{align*}
According to the remark following the statement of Lemma \ref{ANGLES}, we can take $c \leq \frac{1}{\log 2}\cdot\left(7 + \frac{1}{2}\right) \leq 12$. It follows that
\begin{equation*}
\sum_{p \in E_j \atop \theta_p(\tau) \notin \mc{I}} \frac{1}{p^{\sg}} \geq \frac{1}{2} E_j\left(e^{\frac{1}{\sg-1}}\right) + O(1),
\end{equation*}
and, as such, since $1-\cos(\tau \log p) \geq \frac{1}{3}\mu^2$ for $\theta_p(\tau) \notin \mc{I}$,
\begin{equation*}
M_j(\tau,0)  = \left(1+o(1)\right)\sum_{p \in E_j} \frac{1-\cos(\tau \log p)}{p^{\sg}} \geq \frac{1}{3}\mu^2\sum_{p \in E_j \atop \theta_p(\tau) \notin \mc{I}} \frac{1}{p^{\sg}} \geq \frac{1}{6}\mu^2E_j\left(e^{\frac{1}{\sg-1}}\right).
\end{equation*}
The statement now follows by definition.
\end{proof}
We apply the previous lemma to determine the truncation error associated with non-zero $t_j$ and for a larger range of values of $\tau$.
\begin{lem} \label{LemSmallTrunc}
Let $z_j = \rho_je^{it_j}$, with $t_j \in [-\pi,\pi]$, and $s = \sg + i\tau$ with $\tau \in \R$ fixed. Suppose further that $x$ is sufficiently large so that $E_j\left(e^{\frac{1}{\sg-1}}\right) > 2(n+1)$. Let $\nu_j := 1$ if $|t_j| < \frac{1}{192e}$ or $\pi-|t_j| < \frac{1}{192e}$, and $\nu_j = \frac{1}{192e}$ otherwise. Then 
\begin{equation*}
|F(\mbf{z};s)| \leq F(\rho;\sg)e^{-\frac{1}{8}(G(\mbf{t},\tau) + R(\mbf{t},\tau))},
\end{equation*}
where we have defined $R(\mbf{t},\tau) := \frac{1}{100}\sum_{0 \leq j \leq n} k_j\min\left\{1,\frac{E_j\left(e^{\frac{1}{\sg-1}}\right)}{6\log\left(\frac{|\tau|}{\sg-1}\right)}\right\}^2\nu_j^2$ when $\frac{|\tau|}{\sg-1} > 2$ and zero elsewhere, $H(\mbf{t},\tau) := \frac{1}{12}\sum_{0 \leq j \leq n} k_j(t_j-\psi_j(\tau))^2$, and
\begin{equation} \label{CASES}
G(\mbf{t},\tau):= \begin{cases} 
\frac{1}{250}\left(\left(\sum_{0 \leq j \leq n} \rho_j \beta_j\left(e^{\frac{1}{\sg-1}}\right)\right)\frac{\tau^2}{(\sg-1)^2} + \sum_{0 \leq j \leq n} k_jt_j^2\right) &\text{ if $|\tau| < \frac{2}{\sg-1}$} \\ \frac{1}{384e}\left(\sum_{0 \leq j \leq n} \frac{1}{\rho_j\nu_j^2}\right)^{-1}\left(\log\left(1+\min\left\{\frac{1}{\sg-1},\frac{|\tau|}{\sg-1}\right\}^2\right)\right) &\text{ if $|\tau| \geq \frac{2}{\sg-1}$}.\end{cases}
\end{equation}
\end{lem}
\begin{proof}
Consider first the case that $\frac{|\tau|}{\sg-1} < 2$, so that $|\tau \log p| \leq 2$ for each $p \leq e^{\frac{1}{\sg-1}}$.  Note that the inequality $1-\cos u \geq \frac{1}{50}u^2$ holds uniformly for $|u| \leq 2+\pi$. Let $Z := e^{\frac{1}{\sg-1}}$. Then, as $\frac{c_1}{\sg-1} \leq \sum_{p \leq Y} \frac{\log p}{p^{\sg}} \leq \frac{c_2}{\sg-1}$, for $c_2 > c_1 > 0$ constants,
\begin{align*}
&\sum_{p \in E_j} \frac{1-\cos(\tau \log p - t_j)}{p^{\sg}} 
\geq \frac{1}{50}\left(\tau^2 \sum_{p \in E_j \atop p \leq e^{\frac{1}{\sg-1}}} \frac{\log^2 p}{p^{\sg}} - 2\tau t_j \sum_{p \in E_j \atop p \leq e^{\frac{1}{\sg-1}}} \frac{\log p}{p^{\sg}} + t_j^2 E_j\left(e^{\frac{1}{\sg-1}}\right)\right) \\
&\geq \frac{1}{50}\left(\sum_{0 \leq j \leq n} \left(\frac{1}{c_1^2(\sg-1)^2}\rho_j\beta_j(Z)\tau^2 - 2\tau t_j \frac{1}{c_2(\sg-1)}\rho_j\alpha_j(Z) + \rho_j\left(\sum_{p \in E_j} \frac{1}{p^{\sg}}\right)t_j^2\right)\right).
\end{align*}
Now we note that, given a quadratic polynomial of the form $Au^2 -2Buv + Cv^2$, where $A,C > 0$, we have
\begin{equation*}
Au^2 -2Buv + Cv^2 = \frac{1}{A}\left((Au)^2 - 2(Bv)(Au) + (Bv)^2\right) + \left(C-\frac{B^2}{A}\right)v^2 \geq \left(C-\frac{B^2}{A}\right)v^2,
\end{equation*}
and, switching the roles of $u$ and $v$ in the above, we also have $Au^2 - 2Buv + Cv^2 \geq \left(A-\frac{B^2}{C}\right)u^2$. 
Taking $A = \frac{1}{2(\sg-1)^2}\rho_j\beta_j(Z)$, $B := \frac{1}{\sg-1}\rho_j\alpha_j(Z)$ and $C := k_j$, $u = \tau$ and $v = t_j$ for each $j$, we split our sum into two pieces and bound from below using each of the bounds above.  By Lemma \ref{RATIOS} there exists some $\eta > 0$ such that $\alpha_j(Z)^2 \leq \eta \beta_j(Z)$. Hence, we have
\begin{align*}
A-\frac{B^2}{C} &= \frac{1}{c_1^2(\sg-1)^2} \left(\frac{1}{2}\rho_j\beta_j(Z) - \frac{c_1^2}{2c_2E_j\left(e^{\frac{1}{\sg-1}}\right)}\rho_j\alpha_j(Z)^2\right) \\
&\geq \frac{1}{2c_1^2(\sg-1)^2}\rho_j\beta_j(Z) \left(1-\frac{2c_1^2\eta}{c_2E_j\left(e^{\frac{1}{\sg-1}}\right)}\right); \\
C-\frac{B^2}{A} &= k_j - \frac{\rho_j^2\alpha_j(Z)^2}{\rho_j\beta_j(Z)} \geq k_j - \eta\rho_j \geq \frac{1}{2}k_j\left(1-\frac{\eta}{E_j\left(e^{\frac{1}{\sg-1}}\right)}\right).
\end{align*}
Now, 
\begin{align*}
\sum_{p \leq \max_j \rho_j^{\frac{1}{\sg}}} \frac{1-\cos(\tau \log p - t_j)}{p^{\sg}} &\leq \frac{1}{2}\sum_{p \leq \log^{\frac{1}{3}} x} (\tau \log p-t_j)^2 \\
&\leq \frac{1}{2}\left(\tau^2\sum_{p \leq \log^{\frac{1}{3}} x} \frac{\log^2 p}{p^{\sg}} + t_j^2 E_j\left(\log^{\frac{1}{3}} x\right)\right),
\end{align*}
which is half as large as the lower bounds in each of $\tau^2$ and $t_j^2$ that we have just proven. Applying these lower bounds for each $j$ and assuming that $x$ is sufficiently large, we thus have
\begin{align*}
\sum_{0 \leq j\leq n} \rho_j M_j(\tau, t_j) &\geq \frac{1}{250}\sum_{0 \leq j \leq n} \left(\rho_j\beta_j(Z)\frac{\tau^2}{(\sg-1)^2} + k_jt_j^2\right), 
\end{align*}
which proves the estimate for the range $|\tau| < 2(\sg-1)$. \\
Consider now the case where $\frac{|\tau|}{\sg-1} \geq 2$. We remark first that by Cauchy-Schwarz,
\begin{align*}
&\sum_{0 \leq j \leq n} \rho_j \sum_{p \in E_j \atop p > \max_j \rho_j^{\frac{1}{\sg}}} \frac{\sin^2\left(\frac{1}{2}(\tau \log p - t_j)\right)}{p^{\sg}} = \sum_{p > \max_j \rho_j^{\frac{1}{\sg}}} \frac{1}{p^{\sg}}\sum_{0 \leq j \leq n} \rho_j \sin^2\left(\frac{1}{2}(\tau \log p - t_j)\omega_{E_j}(p)\right) \\
&\geq \left(\sum_{0 \leq j\leq n} \frac{1}{\rho_j}\right)^{-1} \sum_{0 \leq j \leq n} \sum_{p \in E_j\atop p > \max_j \rho_j^{\frac{1}{\sg}}} \frac{|\sin(\frac{1}{2}(\tau \log p - t_j))|}{p^{\sg}} \\
&\geq \left(\sum_{0 \leq j\leq n} \frac{1}{\rho_j}\right)^{-1} \sum_{0 \leq j \leq n} \sum_{p \in E_j \atop p > \max_j \rho_j^{\frac{1}{\sg}}} \frac{\sin^2(\frac{1}{2}(\tau \log p - t_j))}{p^{\sg}}.
\end{align*}
In other words, 
\begin{equation} \label{CAUCHY}
\sum_{0 \leq j \leq n} \rho_j M_j(\tau,t_j) \geq \left(\sum_{0 \leq j \leq n} \frac{1}{\rho_j}\right)^{-1}\sum_{0 \leq j \leq n} M_j(\tau,t_j).
\end{equation}
Suppose momentarily that $t_j = 0$, and et $Y$ be a parameter such that $\max_j \rho_j^{\frac{1}{\sg}} < Y \leq e^{\frac{1}{\sg-1}}$. By the Prime Number Theorem with error term from the Korobov-Vinogradov zero-free region (see the notes to chapter II.4 of \cite{Ten2}),
\begin{align*}
&\sum_{0 \leq j \leq n} M_j(\tau,0) = 2\sum_{p \in E_j\atop p > \max_j \rho_j^{\frac{1}{\sg}}} \frac{\sin^2(\frac{1}{2}\tau \log p)}{p^{\sg}} \\
&\geq 2\left(\int_{Y}^{e^{\frac{1}{\sg-1}}} e^{-(\sg-1)\log u} \sin^2\left(\frac{1}{2}\tau \log u\right) \frac{du}{u \log u} + O\left(\int_Y^{\infty} e^{-(\log u)^{\frac{3}{5}+\e} - (\sg-1)\log u} \frac{du}{u} \right)\right) \\
&= 2\left(\int_{\log Y}^{\frac{1}{\sg-1}} e^{-(\sg-1)v} \sin^2\left(\frac{1}{2}\tau v\right) \frac{dv}{v} + O\left(e^{-(\log Y)^{\frac{3}{5}+\e} -(\sg-1)\log Y}\frac{1}{(\sg-1)\log Y}\right)\right).
\end{align*}
In the main term, noting symmetry between $\tau$ and $-\tau$,
\begin{align*}
\int_{\log Y}^{\frac{1}{\sg-1}} e^{-(\sg-1)v} \sin^2\left(\frac{1}{2}\tau v\right) \frac{dv}{v} &= 2\sum_{|\tau|\frac{\log Y}{4\pi} < l \leq \frac{|\tau|}{4\pi(\sg-1)}} \int_0^{\pi} e^{-\frac{2(\sg-1)}{|\tau|}(2\pi l+ v)}\sin^2 v \frac{dv}{v+2\pi l}  \\
&\geq \sum_{|\tau|\frac{\log Y}{4\pi} + 1 < l \leq \frac{|\tau|}{4\pi(\sg-1)}} \frac{e^{-4\pi\frac{(\sg-1)}{|\tau|}l}}{\pi l} \int_0^{\pi} \sin^2 v dv \\
&= \frac{1}{2}\sum_{\frac{|\tau|\log Y}{4\pi} + 1 < l \leq \frac{|\tau|}{4\pi(\sg-1)}} \frac{e^{-4\pi\frac{(\sg-1)}{|\tau|}l}}{l} \geq \frac{1}{4e}\log\left(\frac{1}{(\sg-1)\log Y}\right).
\end{align*}
We note that if $k_j \ll \left(1-\e\right)\frac{\log x}{\log_2 x}$ for $\e \in (0,1)$ (which is surely guaranteed by the assumptions of Theorem \ref{THM1}) then
\begin{equation*}
\frac{k_j}{E_j\left(e^{\frac{1}{\sg-1}}\right)} \ll \log x \ll e^{\frac{1}{\sg-1}},
\end{equation*}
as $(\sg-1) \ll (1-\e)\log_2^{-1} x$ for our choices of $k_j$. Hence, it follows that if $Y := e^{\max\left\{\frac{1}{|\tau|}, \left(\frac{1-\e}{(\sg-1)}\right)\right\}}$ then the estimates for $\rho_j$ in Lemma \ref{PARAMS} show that $Y > \max_j \rho_j$. Note also that the error term from the prime number theorem contributes an $O(1)$ term for this choice because $(\sg-1)\log Y \leq \max\{\frac{\sg-1}{|\tau|},1-\e\} \leq 1$, as $\frac{\sg-1}{|\tau|} < 1$. Finally, 
\begin{equation*}
\log(\frac{1}{(\sg-1)\log Y}) \geq \min\left\{\eta \log\left(\frac{1}{\sg-1}\right), \log\left(\frac{|\tau|}{\sg-1}\right)\right\}.
\end{equation*}
Therefore,
\begin{equation*}
\sum_{0 \leq j \leq n} \sum_{p \in E_j\atop p > \max_j \rho_j^{\frac{1}{\sg}}} \frac{1-\cos(\tau \log p)}{p^{\sg}} \geq \frac{1}{24e} \log\left(\min\left\{\frac{1}{(\sg-1)^2},\left(\frac{\tau}{\sg-1}\right)^2\right\}\right).
\end{equation*}
We now consider the case when the $t_j$ are not necessarily zero.  By \eqref{PROP2}, we have $M_j(\tau,t_j) \geq \sin^2 t_j M_j(2\tau,0)$, and by \eqref{PROP1}, $M_j(\tau,t_j) + M_j(\tau,-t_j) \geq \frac{1}{2}M_j(2\tau,0)$.  Set $C := \frac{1}{192e}$ so that $\nu_j := C$ if $|t_j| > C$ and $\pi-|t_j| > C$, and $\nu_j = 1$ otherwise.  By the mean value theorem, 
\begin{equation*}
|\cos(\tau \log p - t_j)-\cos(\tau \log p + t_j)| \leq |e^{2it_j}-1| = 2|t_j|, 
\end{equation*}
when $|t_j| \leq C$. Therefore, in this case, we have
\begin{align*}
2M_j(\tau,t_j) &\geq M_j(\tau,t_j) + M_j(\tau,-t_j) - \left|M_j(\tau,t_j)-M_j(\tau,-t_j)\right| \\
&\geq \frac{1}{2}M_j(2\tau,0) - \sum_{p \in E_j \atop p > \max_j\rho_j^{\frac{1}{\sg}}} \frac{1}{p^{\sg}}|\cos(\tau \log p - t_j)-\cos(\tau\log p+t_j)| \\
&\geq \frac{1}{2}M_j(2\tau,0) - 3|t_j|E_j\left(e^{\frac{1}{\sg-1}}\right) \geq \frac{1}{2}M_j(2\tau,0) - 3CE_j\left(e^{\frac{1}{\sg-1}}\right).
\end{align*}
Moreover, as $\sin^2 u \geq \frac{2}{\pi} u$ for $|u| < \frac{\pi}{2}$, we get
\begin{align} 
M_j(\tau,t_j) &\geq \frac{1}{2}\max\left\{\frac{1}{2}\sin^2 t_j M_j(2\tau,0), \frac{1}{2}\left(\frac{1}{2}M_j(2\tau,0)-2CE_j\left(e^{\frac{1}{\sg-1}}\right)\right)\right\} \nonumber\\
&\geq \frac{1}{2}\max\left\{\frac{2}{\pi^2}t_j^2M_j(2\tau,0),\frac{1}{2}\left(\frac{1}{2}M_j(2\tau,0)-2CE_j\left(e^{\frac{1}{\sg-1}}\right)\right)\right\} \label{REMOVET},
\end{align}
for any $\tau$. Therefore, combining \eqref{CAUCHY} with \eqref{REMOVET}, we find that when $|\tau| \gg (\sg-1)$,
\begin{align*}
\sum_{0 \leq j \leq n} \rho_jM_j(\tau,t_j) &\geq \frac{1}{2}\left(\sum_{0 \leq j \leq n} \frac{1}{\rho_j\nu_j^2}\right)^{-1} \left(\frac{1}{2}\sum_{0 \leq j \leq n} \left(M_j(2\tau,0)-2CE_j\left(e^{\frac{1}{\sg-1}}\right)\right)\right) \\
&= \frac{1}{384e}\left(\sum_{0 \leq j \leq n} \frac{1}{\rho_j\nu_j^2}\right)^{-1}\log\left(\frac{1}{\sg-1}\right).
\end{align*}
If we define 
\begin{equation*}
\tilde{M}_j(\tau,t_j) := \sum_{p \in E_j \atop \max_j \rho_j < p \leq e^{\frac{|\tau|}{\sg-1}}} \frac{1-\cos(\tau \log p-t_j)}{p^{\sg}}
\end{equation*}
 then by \eqref{PROP1}, 
\begin{equation*}
\tilde{M}_j(\tau,t_j) + \tilde{M}_j(\tau,-t_j) \geq \sum_{p \in E_j \atop \max_j < p \leq e^{\frac{|\tau|}{\sg-1}}} \frac{1-\cos(2\tau \log p)}{p^{\sg}};
\end{equation*}
thus, by the same argument as above with $E_j\left(e^{\frac{|\tau|}{\sg-1}}\right)$ in place of $E_j\left(e^{\frac{1}{\sg-1}}\right)$, we get (since $M_j(\tau,t_j) \geq \tilde{M}_j(\tau,t_j)$, obviously)
\begin{align*}
\sum_{0 \leq j \leq n} \rho_jM_j(\tau,t_j) &\geq \frac{1}{4}\left(B(\mbf{t})\sum_{0 \leq j \leq n} \sum_{p \in E_j \atop \max_j \rho_j^{\frac{1}{\sg}} < p \leq e^{\frac{|\tau|}{\sg-1}}} \frac{1-\cos(2\tau \log p)}{p^{\sg}}-4CE_j\left(e^{\frac{|\tau|}{\sg-1}}\right)\right) \\
&= \frac{1}{384e}B(\mbf{t})\log\left(\frac{|\tau|}{\sg-1}\right),
\end{align*}
where $B(\mbf{t}):= \left(\sum_{0 \leq j \leq n} \frac{1}{\rho_j\nu_j^2}\right)^{-1}$. In the particular case that $2 < \frac{|\tau|}{\sg-1} \ll 1$, combining \eqref{REMOVET} with the conclusion of Lemma \ref{EXTRADECAY} and the definition of $\nu_j$ gives
\begin{align*}
\sum_{0 \leq j \leq n} \rho_jM_j(\tau,t_j) &\geq \frac{1}{4}\left(\sum_{0 \leq j \leq n} \rho_j\nu_j^2 \left(M_j(2\tau,0)-2CE_j\left(e^{\frac{1}{\sg-1}}\right)\right)\right) \\
&\geq \frac{1}{4}\left(\frac{1}{24}-2C\right)\sum_{0 \leq j \leq n} \rho_j\nu_j^2E_j\left(e^{\frac{1}{\sg-1}}\right).
\end{align*}
By Lemmas \ref{LemUnif} and \ref{PARAMS}, $\rho_jE_j\left(e^{\frac{1}{\sg-1}}\right) \geq \frac{1}{2}k_j$, so the proof is complete upon identifying this last term as $R\left(\mbf{t},\tau)\right)$, and using $2\sum_{0 \leq j \leq n} \rho_jM_j(\tau,t_j) \geq G(\mbf{t},\tau) + R(\mbf{t},\tau)$.
\end{proof}
As described earlier, we shall proceed with estimating \eqref{PerrId} by approximating it by an integral on an $(n+2)$-dimensional box of the form $[-T,T] \times \prod_{0 \leq j \leq n} [-\theta_j,\theta_j]$, where $T,\theta_0,\ldots,\theta_n > 0$ are parameters that we shall choose (see Lemma \ref{LemInt}).  To determine how suitable of an approximation this will be, we shall first calculate the integral contributed by this box, after which we will have a frame of reference for determining which error terms are acceptable.  It will be necessary first to determine the approximate dimensions of the box in question. The following lemma provides us with a hint at choosing these data.
\begin{lem} \label{ERR}
Let $T,\theta_0,\theta_1,\ldots,\theta_n > 0$. Let $h(\mbf{z};s)$ be defined as in \eqref{hDef}. Let $z_j = \rho_je^{it_j}$ for each $0 \leq j \leq n$ and $s = \sg + i\tau$, where $|t_j| \leq \theta_j$, and $|\tau| \leq T$. Set $\delta := \frac{T}{\sg-1}$, and assume that $\delta < 1$. Then
\begin{equation}
h(\mbf{z};s) =  -\frac{\tau^2}{2} \sum_{0 \leq j \leq n} \rho_j \left(\sum_{p \in E_j} \frac{\log^2 p}{p^{\sg}} + O\left(\sum_{0 \leq j \leq n} Q_j\right)\right) \label{HZS},
\end{equation}
where
\begin{equation*}
Q_j := \frac{\theta_j^2}{\log(\rho_j+1)} + T(T+\theta)\log(\rho_j+1) + \delta^3\gamma_j +\theta_j\alpha_j\delta.
\end{equation*}
\end{lem}
\begin{proof}
Recall that $h$ is defined as 
\begin{align*}
h(\mbf{z};s) &= \sum_{0 \leq j \leq n} \int_{\rho_j}^{z_j} dw_j (z_j-w_j)f_{z_jz_j}(w_j;\sg) \\
&+ i\tau \sum_{0 \leq j \leq n} \int_{\rho_j}^{z_j}dw_jf_{z_js}(w_j;\sg) -\int_0^{\tau}d\tau'(\tau-\tau')f_{ss}(\mbf{z};\sg + i\tau') \\
&=: \sum_{0 \leq j \leq n} \left(I_j + i\tau II_j -III_j\right).
\end{align*}
We first estimate $II_j$ for each $0 \leq j \leq n$. Using \eqref{ZiS}, we find that
\begin{align}
II_j &= -\sum_{p \in E_j} p^{\sg} \log p \int_{\rho_j}^{z_j} \frac{dw_j}{(p^{\sg}-1+w_j)^2} = -(z_j-\rho_j)\sum_{p \in E_j} \frac{p^{\sg} \log p}{(p^{\sg}-1+\rho_j)(p^{\sg}-1 + z_j)} \nonumber\\
&= -(z_j-\rho_j)\sum_{p \in E_j} \left(\frac{\log p}{p^{\sg}-1+\rho_j} - (z_j-1)\frac{\log p}{(p^{\sg}-1+\rho_j)(p^{\sg}-1+z_j)}\right) \\
&= -(z_j-\rho_j)\sum_{p \in E_j} \left(\frac{\log p}{p^{\sg}} - \log p\left(\frac{\rho_j-1}{p^{\sg}(p^{\sg}-1+\rho_j)} + \frac{z_j-1}{(p^{\sg}-1+\rho_j)(p^{\sg}-1+z_j)}\right)\right)  \label{EST} \\
&=: -(z_j-\rho_j)\left(\sum_{p \in E_j} \frac{\log p}{p^{\sg}} - (T_1+T_2)\right).
\end{align}
Let $u > 0$. Observe that by Lemma \ref{SIMPLE}
\begin{align}
u\sum_{p \in E_j \atop p \leq u} &\frac{\log p}{p^{\sg}(p^{\sg}+u)} \ll \sum_{p \leq u} \frac{\log p}{p} \ll \log(1+u) \label{S1}\\
u\sum_{p \in E_j \atop p > u} &\frac{\log p}{p^{\sg}(p^{\sg}+u)} \ll \rho_j\sum_{p > u} \frac{\log p}{p^{2\sg}} \ll 1 \label{S2}.
\end{align}
Applying these two estimates with $u = \rho_j-1$ and $u := |z_j-1| \leq \rho_j-1$ gives an upper bound for $T_1$ and $T_2$, respectively. Hence, inserting these estimates into the two series in the brackets of \eqref{EST} and using $|z_j-\rho_j| = \rho_j|t_j|$ (by the mean value theorem), we get
\begin{equation}
II_j = -(z_j-\rho_j)\sum_{p \in E_j} \frac{\log p}{p^{\sg}} + O\left(\rho_j\theta_j \log (\rho_j+1)\right). \label{IIj}
\end{equation}
Next, we estimate $I_j$.  Integrating by parts, we have
\begin{align*}
I_j &= \sum_{p \in E_j} \int_{\rho_j}^{z_j} dw_j\frac{z_j-w_j}{(p^{\sg}-1+w_j)^2} = \sum_{p \in E_j} \left(\frac{z_j-\rho_j}{p^{\sg}-1+\rho_j} - \int_{\rho_j}^{z_j} \frac{dw_j}{p^{\sg}-1+w_j}\right) \\
&= \sum_{p \in E_j} \left(\frac{z_j-\rho_j}{p^{\sg}-1+\rho_j} - \log\left(1+\frac{z_j-\rho_j}{p^{\sg}-1+\rho_j}\right)\right),
\end{align*}
where we continue to use the principal branch of logarithm. Taylor expanding the logarithm, we get
\begin{equation*}
I_j = \frac{1}{2}(z_j-\rho_j)^2\sum_{p \in E_j} \frac{1}{(p^{\sg}-1+\rho_j)^2} + O\left(\rho_j^3\theta_j^3 \sum_{p \in E_j} \frac{1}{(p^{\sg}-1+\rho_j)^3}\right).
\end{equation*}
From \eqref{P2} we get that for $l \in \left\{2,3\right\}$,
\begin{align*}
\sum_{p \in E_j} \frac{1}{(p^{\sg}-1+\rho_j)^l} &\ll \frac{1}{\rho_j^{l-1} \log (\rho_j+1)},
\end{align*}
so that
\begin{equation}
I_j \ll \frac{\rho_j \theta_j^2}{\log(\rho_j+1)}. \label{Ij}
\end{equation}
Finally, we estimate $III_j$.  First, as
\begin{align*}
f_s(\mbf{z};s) &= -\sum_{0\leq j \leq n} z_j\sum_{p\in E_j} \frac{p^s \log p}{(p^s-1)(p^s-1+z_j)} \\
&= -\sum_{0 \leq j \leq n}z_j \sum_{p \in E_j}\left( \frac{\log p}{p^s-1} - (z_j-1)\frac{\log p}{(p^s-1)(p^s-1+z_j)}\right),
\end{align*}
differentiating with respect to $s$ again yields
\begin{align*}
f_{ss}(\mbf{z};s) &= \sum_{0 \leq j \leq n} z_j\sum_{p\in E_j} \left(\frac{p^s \log^2 p}{(p^s-1)^2} - (z_j-1)\frac{p^s \log^2 p}{(p^s-1)(p^s-1+z_j)}\left(\frac{1}{p^s-1} + \frac{1}{p^s-1+z_j}\right)\right) \\
&= \sum_{0 \leq j \leq n} z_j\sum_{p \in E_j} (\frac{\log^2 p}{p^s} + \log^2 p\left(\frac{1}{p^s(p^s-1)} + \frac{1}{(p^s-1)^2}\right) \\
&-(z_j-1)\left(\frac{1}{p^s-1} + \frac{1}{p^s-1+z_j}\right)\left(\frac{\log^2 p}{p^s-1+z_j} + \frac{\log^2 p}{(p^s-1)(p^s-1+z_j)}\right)) \\
&= \sum_{0\leq j \leq n} z_j \sum_{p \in E_j} \frac{\log^2 p}{p^s} + O\left(\sum_{0\leq j \leq n} \rho_j \log^2(\rho_j+1)\right),
\end{align*}
the last estimate following from \eqref{S1} and \eqref{S2} (with $\log p$ replaced by $\log^2 p$). Returning to $III_j$, for each prime $p$, applying integration by parts as before,
\begin{align*}
\int_0^{\tau} d\tau' (\tau-\tau')e^{-i\tau'\log p} = -\frac{i\tau}{\log p} + \frac{i}{\log p} \int_0^{\tau} e^{-i\tau' \log p} d\tau' = -\frac{i\tau}{\log p} + \frac{1}{\log^2 p} (1-e^{-i\tau \log p}).
\end{align*}
Hence, summing over all primes and applying Fubini's theorem,
\begin{equation*}
\int_0^{\tau} d\tau'(\tau-\tau')f_{ss}(\mbf{z};\sg + i\tau') = \sum_{p \in E_j} \frac{1 - i\tau \log p - e^{-i\tau \log p}}{p^{\sg}} + O\left(\log(\rho_j+1)T^2\right) .
\end{equation*}
It follows that
\begin{equation}
III_j = z_j\sum_{p \in E_j} \frac{1-i\tau \log p - e^{-i\tau \log p}}{p^{\sg}} + O\left(\rho_j\log(\rho_j+1)T^2\right). \label{IIIj}
\end{equation}
Combining \eqref{IIj}, \eqref{Ij} and \eqref{IIIj}, we have
\begin{align*}
h(\mbf{z};s) = i\tau \sum_{0 \leq j \leq n} \rho_j \sum_{p \in E_j} \frac{\log p}{p^{\sg}} &-\sum_{0 \leq j \leq n} (\rho_j + (z_j-\rho_j)) \sum_{p \in E_j} \frac{1-e^{-i\tau \log p}}{p^{\sg}} \\
&+ O\left(\sum_{0 \leq j \leq n} \rho_j\left(\frac{\theta_j^2}{\log(\rho_j+1)} + 	T(T+\theta)	\log(\rho_j+1)\right)\right).
\end{align*}
To treat the term in $z_j-\rho_j$ above, we observe that by the mean value theorem,
\begin{align*}
\sum_{0 \leq j \leq n} \rho_j\theta_j \sum_{p \in E_j} \frac{|1-e^{-i\tau \log p}|}{p^{\sg}} &\leq |\tau|\sum_{0 \leq j \leq n} \rho_j \theta_j \sum_{p \in E_j} \frac{\log p}{p^{\sg}} \\
&= \frac{|\tau|}{\sg-1} \sum_{0 \leq j \leq n} \rho_j \alpha_j\theta_j \leq \delta \sum_{0 \leq j \leq n} \rho_j \alpha_j \theta_j.
\end{align*}
Since $|z_j-\rho_j| \leq \rho_j\theta_j$ in our range of arguments $t_j$, we are left with
\begin{equation*}
h(\mbf{z};s) = -\sum_{0 \leq j \leq n} \rho_j \sum_{p \in E_j} \frac{1-i\tau \log p-e^{-i\tau \log p}}{p^{\sg}} + O\left(\sum_{0 \leq j \leq n} \rho_j\left(Q_j-\delta^3\gamma_j\right)\right).
\end{equation*}
Hence, expanding in powers of $\tau \log p$ the expression $1-i\tau \log p - e^{-i\tau \log p}$ in this last equation, via \eqref{POWERS}, the remainder term (or order 3) in Taylor's theorem becomes
\begin{equation*}
|\tau|^3\sum_{0 \leq j \leq n} \rho_j \sum_{p \in E_j} \frac{\log^3 p}{p^{\sg}} \leq \delta^3\sum_{0 \leq j \leq n} \rho_j \gamma_j.
\end{equation*}
It follows that
\begin{equation*}
h(\mbf{z};s) = -\frac{\tau^2}{2}\sum_{0 \leq j \leq n} \rho_j\sum_{p \in E_j} \frac{\log^2 p}{p^{\sg}} + O\left(\sum_{0 \leq j \leq n} \rho_j Q_j \right).
\end{equation*}
\end{proof}
In order for $\int_{-\theta_j}^{\theta_j} e^{-\frac{1}{2}kt_j^2}$ in \eqref{NEWPERR} to be computable with good error, it seems suitable to select $\theta_j$ such that $\theta_j^2k_j \ra \infty$ as $x \ra \infty$ (in the regime where $k_j$ is an increasing function of $x$), in order to replace the short interval integral with the full Gaussian integral $\int_{-\infty}^{\infty} e^{-\frac{1}{2}k_jt_j^2} dt_j$. On the other hand, we must also ensure that $\frac{\rho_j}{\log \rho_j}\theta_j^2$ is small, so that the error term in Lemma \ref{ERR} is small as well. Our assumption that $E_j(x) \ra \infty$ as $x \ra \infty$ for each $j$ will be necessary here in order for these two conditions to be possible. Moreover, we will also need $\delta = \frac{|\tau|}{\sg-1}$ to be small in order for $\delta\sum_{0 \leq j \leq n} \rho_j \alpha_j\theta_j$ and $\delta\sum_{0 \leq j \leq n} \rho_j\gamma_j$ to be small; it cannot be too small, however, as $\delta \sum_{0 \leq j \leq n} \rho_j\beta_j$ must be large, in order for $\int_{-T}^T e^{h(\mbf{z};s)} \frac{x^s}{s}ds$ to be extended to an infinite integral. This motivates the choices that we shall make in Lemma \ref{LemInt}. \\
Set $\mc{B} := [-\theta_0,\theta_0] \times \cdots \times [-\theta_n,\theta_n]$.
\begin{lem} \label{LemInt}
Let $\eta, \e \in \left(0,\frac{1}{9}\right)$. Assume one of $\mathbf{H}_1(\sg)$, $\mathbf{H}_2(\sg)$ or $\mathbf{H}_3(\sg)$. Suppose that $E_j\left(e^{\frac{1}{\sg-1}}\right)^{m_j} \ll k_j \ll \min\{k_j^{\max},\log^{\frac{2}{3}-\mu} x\}$ for each $0 \leq j \leq n$, where $\mu > 0$ is arbitrary but fixed. For each $j$ let $\theta_j := k_j^{-\frac{1}{2}}E_j\left(e^{\frac{1}{\sg-1}}\right)^{\eta}$. Denote by $j_0$ the index of the set $E_j$ such that $\rho_j^2\beta_j\theta_j^2$ is maximal for $j = j_0$. Let $\delta := E_{j_0}\left(e^{\frac{1}{\sg-1}}\right)^{\e}\left(\sum_{0 \leq j \leq n} \rho_j\beta_j\right)^{-\frac{1}{2}}$, and set $T := \delta(\sg-1)$.  
%
%
Finally, define
\begin{equation*}
R_j := \theta_jk_j^{\frac{1}{2}} e^{-k_j\theta_j^2/2} + k_j\theta_j^3 + \rho_j \left(\frac{\theta_j^2}{\log(\rho_j+1)} + T(T+\theta_j)\log(\rho_j+1) + \delta^3\gamma_j + \theta_j\alpha_j\delta\right).
\end{equation*}
Then 
\begin{align*}
I &:= \int_{-T}^T \frac{d\tau}{\sg + i\tau} \int_{\mc{B}} d\mbf{t} e^{-\frac{1}{2}\sum_{0 \leq j \leq n} k_jt_j^2} H(\rho e^{i\mbf{t}};\sg + i\tau) \\
&= \left(\frac{2\pi}{\sum_{0 \leq j \leq n} \rho_j\beta_j}\right)^{\frac{1}{2}}(\sg-1)\prod_{0 \leq j \leq n} \left(\frac{2\pi}{k_j}\right)^{\frac{1}{2}}\left(1+O(R_j)\right),
\end{align*}
where each $R_j$ satisfies the estimate
\begin{equation} \label{ERRS}
R_j \ll_{\eta,\e} E_j\left(e^{\frac{1}{\sg-1}}\right)^{-\frac{1}{2}+\e+\eta}.
\end{equation} 
\end{lem}
\begin{proof}
We first remark that $\delta < 1$, as needed in Lemma \ref{ERR}. Indeed, since $\sum_{0 \leq j \leq n} \beta_j \sim 1$ as $x \ra \infty$, there must exist some $j_1$ such that $\beta_{j_1} \gg_n 1$. By choice of $j_0$, 
\begin{equation*}
\rho_{j_0}^2\beta_{j_0}\theta_{j_0}^2 = \rho_{j_0}\beta_{j_0}E_{j_0}\left(e^{\frac{1}{\sg-1}}\right)^{-(1-\eta)}
\end{equation*}
is maximal among $j$, and thus, as $\rho_{j_1} \gg E_{j_1}\left(e^{\frac{1}{\sg-1}}\right)$,
\begin{equation*}
\rho_{j_0}\beta_{j_0} E_{j_0}\left(e^{\frac{1}{\sg-1}}\right)^{-(1-\eta)} \gg_n \rho_{j_1}E_{j_1}\left(e^{\frac{1}{\sg-1}}\right)^{-(1-\eta)} \gg_n E_{j_1}\left(e^{\frac{1}{\sg-1}}\right)^{\eta}.
\end{equation*}
Thus, $\rho_{j_0}\beta_{j_0} \gg_n E_{j_0}\left(e^{\frac{1}{\sg-1}}\right)^{1-\eta}E_{j_1}\left(e^{\frac{1}{\sg-1}}\right)^{\eta}$, which implies that
\begin{equation} \label{DELTAEST}
\delta \asymp E_{j_0}\left(e^{\frac{1}{\sg-1}}\right)^{\e} \left(\sum_{0 \leq j \leq n} \rho_j\beta_j\right)^{-\frac{1}{2}} \leq \rho_{j_0}^{-\frac{1}{2}}\beta_{j_0}^{-\frac{1}{2}} E_{j_0}\left(e^{\frac{1}{\sg-1}}\right)^{\e} \ll E_{j_0}\left(e^{\frac{1}{\sg-1}}\right)^{\e-\frac{1}{2}(1-\eta)} E_{j_1}^{-\frac{1}{2}\eta}.
\end{equation}
As we are assuming that $\e + \eta < \frac{1}{2}$, we clearly see that $\delta = o(1)$, and the claim above trivially follows for sufficiently large $x$. \\
According to the above remark we may apply Lemma \ref{ERR}, which gives
\begin{align*}
&e^{-\frac{1}{2}\sum_{0 \leq j \leq n} k_jt_j^2} H(\rho e^{i\mbf{t}};\sg + i\tau) = \text{exp}\left(-\frac{1}{2}\tau^2\sum_{0 \leq j \leq n} \rho_j\sum_{p \in E_j} \frac{\log^2 p}{p^{\sg}} -\sum_{0 \leq j \leq n} \frac{1}{2}k_jt_j^2 + O\left(R_j'\right)\right) \\
&= \text{exp}\left(-\frac{1}{2}\tau^2\sum_{0 \leq j \leq n} \rho_j\sum_{p \in E_j} \frac{\log^2 p}{p^{\sg}} -\sum_{0 \leq j \leq n} \frac{1}{2}k_jt_j^2\right)\prod_{0 \leq j \leq n} \left(1+ O\left(R_j'\right)\right),
\end{align*}
where $R_j'$ is defined such that $R_j =: R_j' + \theta_jk_j^{\frac{1}{2}}e^{-k_j\theta_j^2/2}$. \\
In estimating $I$ we proceed first with the integrals over $t_j$ for each $j$. Letting $\e > 0$ be a parameter to be specified, we have
\begin{align}
\int_{-\theta_j}^{\theta_j} e^{-\frac{1}{2}k_jt_j^2} dt_j &= \left(\frac{2}{k_j}\right)^{\frac{1}{2}}\left(\int_{-\infty}^{\infty} e^{-u^2} du + O\left(\int_{|u| > (k_j/2)^{\frac{1}{2}}\theta_j} e^{-u^2}du \right)\right) \nonumber\\
&= \sqrt{\frac{2\pi}{k_j}}\left(1 + O\left(e^{-\frac{1}{2}(1-\e)k_j\theta_j^2}\int_{-\infty}^{\infty} e^{-\e u^2} du\right)\right)  \\
&= \sqrt{\frac{2\pi}{k_j}}\left(1 + O\left(e^{-\frac{1}{2}(1-\e)k_j\theta_j^2}\e^{-\frac{1}{2}}\right)\right) \label{OPT}.
\end{align}
The optimal value for $\e$ is $\frac{1}{k_j\theta_j^2}$, whence we get an error term $O\left(\theta_jk_j^{\frac{1}{2}} e^{-\frac{1}{2}k_j\theta_j^2}\right)$, for each $j$.  Since $e^{-\frac{1}{2}k_j\theta_j^2} \ll e^{-\frac{1}{2}E_j\left(e^{\frac{1}{\sg-1}}\right)^{\eta}}$, it follows that $\theta_jk_j^{\frac{1}{2}} e^{-\frac{1}{2}k_j\theta_j^2}  \ll E_j\left(e^{\frac{1}{\sg-1}}\right)^{-\frac{1}{2}}$ for $x$ sufficiently large (in terms of $\eta$). \\
Consider now the integral in $\tau$. First, note that $\left|\frac{1}{\sg + i\tau} - 1\right| \leq \frac{|\sg - 1| + T}{1-T} \ll \frac{1}{\sg-1}$ (the last inequality following because $\delta < 1$). Set 
\begin{equation*}
D:= \sum_{0 \leq j \leq n} \rho_j \sum_{p \in E_j} \frac{\log^2 p}{p^{\sg}} = \frac{1}{(\sg-1)^2}\sum_{0 \leq j \leq n} \rho_j\beta_j.
\end{equation*}
The integral in $\tau$ can thus be simplified as
\begin{align}
\int_{-T}^T e^{-\frac{1}{2}D\tau^2} \frac{d\tau}{\sg + i\tau} &= \left(1+O\left(\sg-1\right)\right) \left(\frac{2}{D}\right)^{\frac{1}{2}}\int_{-T\left(\frac{D}{2}\right)^{\frac{1}{2}}}^{T\left(\frac{D}{2}\right)^{\frac{1}{2}}} e^{-u^2}{du}. \label{FIRSTSMALL}\\
\end{align}
By our choice of $\delta$,
\begin{align*}
T^2D &= \left(\frac{T}{\sg-1}\right)^2\sum_{0 \leq j \leq n} \rho_j\beta_j = \left(\delta\sum_{0 \leq j \leq n} \rho_j\beta_j\right)^2\gg E_{j_0}\left(e^{\frac{1}{\sg-1}}\right)^{2\e}
\end{align*}
for each $j$. As such, we can replace the integral over the segment $[-TD^{\frac{1}{2}},TD^{\frac{1}{2}}]$ with an integral over all of $\mb{R}$ as in \eqref{OPT}. Thus,
\begin{equation}
\int_{-T\left(\frac{D}{2}\right)^{\frac{1}{2}}}^{T\left(\frac{D}{2}\right)^{\frac{1}{2}}} e^{-u^2} du = \sqrt{2\pi} + O\left(TD^{\frac{1}{2}}e^{-\frac{1}{2}T^2D}\right) = \sqrt{2\pi} + O_{\e}\left(E_{j_0}\left(e^{\frac{1}{\sg-1}}\right)^{-\frac{1}{2}}\right) \label{SECONDSMALL},
\end{equation}
in analogy to the derivation of \eqref{OPT}, for $x$ sufficiently large in terms of $\e$ as well. As $1+O\left(E_{j_0}\left(e^{\frac{1}{\sg-1}}\right)^{-\frac{1}{2}}\right) \leq \prod_{0 \leq j \leq n} \left(1+O\left(E_j\left(e^{\frac{1}{\sg-1}}\right)^{-\frac{1}{2}}\right)\right)$, it follows that 
\begin{align*}
I = \left(1+ O\left((\sg - 1) \right)\right)(\sg-1) \left(\frac{2\pi}{\sum_{0 \leq j \leq n} \rho_j\beta_j}\right)^{\frac{1}{2}}\prod_{0 \leq j \leq n} \left(\frac{2\pi}{k_j}\right)^{\frac{1}{2}}\left(1+O(R_j)\right).
\end{align*}
It remains to prove \eqref{ERRS}. 
To start, we see by definition of $\theta_j$ that
\begin{align*}
k_j\theta_j^3 &\asymp k_j^{-\frac{1}{2}}E_j\left(e^{\frac{1}{\sg-1}}\right)^{3\eta} \ll E_j\left(x\right)^{-1+3\eta}; \\
\frac{\rho_j}{\log(\rho_j+1)}\theta_j^2 &\ll \rho_jk_j^{-1}E_j\left(e^{\frac{1}{\sg-1}}\right)^{2\eta} \ll E_j\left(e^{\frac{1}{\sg-1}}\right)^{-1+2\eta},
\end{align*}
for each $j$. Next, let $j''$ be the index that maximizes $r_j := \rho_j \gamma_j$ among all $j$ (as in $\mathbf{H}_2(\sg)$). Note that 
\begin{equation*}
\left(\sum_{0 \leq j \leq n} \rho_j \beta_j\right)^{\frac{3}{2}} \geq \left(\rho_{j''}^2\beta_{j''}^2\sum_{0 \leq j \leq n} \rho_j\beta_j\right)^{\frac{1}{2}} = \rho_{j''}\beta_{j''}\left(\sum_{0 \leq j \leq n} \rho_j\beta_j\right)^{\frac{1}{2}},
\end{equation*}
so we immediately find, by our choice of $j''$, that 
\begin{align}
\delta^3\sum_{0 \leq j \leq n} \rho_j\gamma_j &\ll_n E_{j_0}\left(e^{\frac{1}{\sg-1}}\right)^{3\e}\left(\rho_{j''}\beta_{j''} \left(\sum_{0 \leq j \leq n} \rho_j\beta_j\right)^{\frac{1}{2}}\right)^{-1}\rho_{j''}\beta_{j''}\left(\frac{\gamma_{j''}}{\beta_{j''}}\right) \nonumber\\
&\ll_n E_{j_0}\left(e^{\frac{1}{\sg-1}}\right)^{3\e}\left(\sum_{0 \leq j \leq n} \rho_j\beta_j\right)^{-\frac{1}{2}}\left(\frac{\gamma_{j''}}{\beta_{j''}}\right) \label{IMPROVABLE}.
\end{align}
Assume $\mathbf{H}_1(\sg)$ first.  Then we can furthermore assume that $j_1$, chosen above, satisfies $\beta_{j'} \gg_n 1$ as well, because $\alpha_{j'} \gg_n 1$, and $\beta_{j'} \gg \alpha_{j'}^2$ by Lemma \ref{RATIOS}. Hence, $\rho_{j'}\beta_{j'} \gg_n \log^3\left(\frac{1}{\sg-1}\right)$ when $m_{j'} = 4$. Also, by Lemma \ref{RATIOS}, $\gamma_j \ll \beta_j\log\left(\frac{1}{\beta_j}\right) \ll \beta_j\log\left(\frac{1}{\sg-1}\right)$ for each $j$, since $\beta_j \gg (\sg-1)^2$. It follows that
\begin{align*}
E_{j_0}\left(e^{\frac{1}{\sg-1}}\right)^{3\e}\left(\sum_{0 \leq j \leq n} \rho_j\beta_j\right)^{-\frac{1}{2}}\left(\frac{\gamma_{j''}}{\beta_{j''}}\right) &\ll E_{j_0}\left(e^{\frac{1}{\sg-1}}\right)^{3\e}\frac{\log\left(\frac{1}{\sg-1}\right)}{(\rho_{j'}\beta_{j'})^{\frac{1}{2}}} \ll \log^{-\frac{1}{2}+3\e}\left(\frac{1}{\sg-1}\right),
\end{align*}
which suffices. \\
Next, assume $\mathbf{H}_2(\sg)$. Then $\gamma_{j''}/\beta_{j''} \ll_n 1$, and by \eqref{DELTAEST}, $\delta \ll E_{j_0}\left(e^{\frac{1}{\sg-1}}\right)^{-\frac{1}{2}+\frac{1}{2}\eta+\e}$. Thus, $\delta^3\sum_{0 \leq j \leq n} \rho_j\gamma_j \ll E_{j_0}\left(e^{\frac{1}{\sg-1}}\right)^{-\frac{1}{2}(1-\eta)+ 4\e}$, which also suffices.  \\
Finally, by Lemma \ref{RATIOS}, $\mathbf{H}_3(\sg)$ implies $\mathbf{H}_2(\sg)$, so the same conclusion as the one just proved still follows. \\
Recall again that $\alpha_j^2 \ll \beta_j$ for each $j$ by Lemma \ref{RATIOS}. Thus, by our choice of $j_0$,
\begin{align*}
\delta^2\left(\sum_{0 \leq j \leq n} \rho_j \alpha_j \theta_j\right)^2 &\ll_n E_{j_0}\left(e^{\frac{1}{\sg-1}}\right)^{2\e}\left(\sum_{0 \leq j \leq n} \rho_j \beta_j\right)^{-1}\left(\sum_{0 \leq j \leq n} \rho_j^2\alpha_j^2\theta_j^2\right) \\
&\ll E_{j_0}\left(e^{\frac{1}{\sg-1}}\right)^{2\e}\left(\sum_{0 \leq j \leq n} \rho_j \beta_j\right)^{-1}\left(\sum_{0 \leq j \leq n} \rho_j^2\beta_j\theta_j^2\right) \\
&\ll_n  E_{j_0}\left(e^{\frac{1}{\sg-1}}\right)^{2\e} \frac{\rho_{j_0}^2\beta_{j_0}\theta_{j_0}^2}{\rho_{j_0}\beta_{j_0}} =E_{j_0}\left(e^{\frac{1}{\sg-1}}\right)^{2\e} \rho_{j_0}\theta_{j_0}^2 \\
&\ll k_{j_0}\theta_{j_0}^2 E_{j_0}\left(e^{\frac{1}{\sg-1}}\right)^{-1+2\e} \ll E_{j_0}\left(e^{\frac{1}{\sg-1}}\right)^{-1+2(\e+\eta)}.
\end{align*}
Since $\eta+\e < \frac{1}{4}$, this last bound gives $\delta\sum_{0 \leq j \leq n} \rho_j\alpha_j\theta_j \ll E_{j_0}\left(e^{\frac{1}{\sg-1}}\right)^{-\frac{1}{2} +(\eta + \e)}$.
Lastly, we note that $T \ll \theta_j^2E_j(x)^{-1}$ for each $j$, and $k_j \ll \log^{\frac{2}{3}-\mu} x$.  Indeed, by Lemma \ref{PARAMS}, we see that
\begin{equation*}
\sg-1\leq \frac{1}{\log x}\sum_{0 \leq j \leq n} \rho_j \ll \frac{1}{\log x}\sum_{0 \leq j \leq n} k_j \ll \log^{-\frac{1}{3}-\mu} x,
\end{equation*}
so that $(\sg-1)k_j^{\frac{1}{2}} \leq \log^{-\mu} x$ uniformly in our range of $k_j$.  It follows that $(\sg-1)\theta_j^{-1} \ll \log^{-\mu/2} x$, whence also that $T < (\sg-1) \ll \log^{-\mu/2} x\theta_j$. 
As $\log(\rho_j+1) \ll \log\left(\frac{1}{\sg-1}\right)$,
\begin{equation*}
\rho_jT(T+\theta)\log(\rho_j+1) \ll \rho_j\theta_j^2\log^{-\mu/2} x \log_2 x \ll_{\mu} E_j\left(e^{-\frac{1}{\sg-1}}\right)^{-1}. 
\end{equation*}
Since, now, $(\sg-1) \ll E_j\left(e^{\frac{1}{\sg-1}}\right)$ for each $j$, we can absorb the factor $1+O\left(\sg-1\right)$ into one of the factors $1+O\left(R_j\right)$. Replacing $\e$ by $\frac{1}{4}\e$ completes the proof.
\end{proof}
We now consider the truncation error associated with restricting our initial Perron integral \eqref{PerrId} to the $(n+2)$-dimensional box $[-T,T] \times \prod_{0 \leq j \leq n} [-\theta_j,\theta_j]$, using the choices of parameters in Lemma \ref{LemInt}. The first part of the proof is inspired directly from \cite{HiT}; we give a detailed proof of the modification of it that is necessary to us.
\begin{lem} \label{LemAppTrunc}
Let $x \geq 3$, $A \geq 2$, set $L := \log^A x$, and let $T$, $\delta$, and $\theta_0,\ldots,\theta_n$ be fixed as in Lemma \ref{LemInt}.  Assume that $k_j \geq E_j(x)^2$ for each $0 \leq j \leq n$. Then, uniformly in $x$,
\begin{equation*}
\pi(x;\mbf{E};\mbf{k}) = x^{\sg}F(\mbf{\rho};\sg)\left(\left(\prod_{0 \leq j \leq n} \rho_j^{-k_j}\left(1+O(S_j)\right)\right)\frac{I}{(2\pi)^{n+2}} + O\left(\frac{1}{\log^B x}\right)\right)
\end{equation*}
where, as before, $\mbf{\rho}e^{i\mbf{t}} := (\rho_0e^{it_0},\ldots,\rho_ne^{it_n})$, $B := A-\frac{3}{7}$, and $S_j := E_j\left(e^{\frac{1}{\sg-1}}\right)^{-\frac{1}{2}}$.
\end{lem}
\begin{proof}
Restricting to $|z_j| = \rho_j$ for each $0\leq j \leq n$, the Dirichlet series $\sum_{m \geq 1} \frac{z_0^{\omega_{E_0(m)}}\cdots z_n^{\omega_{E_n(m)}}}{m^s}$ has coefficients bounded above by $a(m) := \prod_{0 \leq j \leq n} \rho_j^{\omega_{E_j}(m)}$ for each $m \geq 1$ (note that $\sum_{m \geq 1} \frac{a(m)}{m^{\sg}} = F(\mbf{\rho};\sg)$). Thus, by Theorem II.2.4 in \cite{Ten2}, we have
\begin{equation*}
\frac{1}{2\pi}\int_{-\infty}^{\infty} F(\mbf{z};\sg + i\tau) x^{\sg + i\tau} \frac{d\tau}{\sg + i\tau} = \frac{1}{2\pi}\int_{-L}^LF(\mbf{z};\sg + i\tau) x^{\sg + i\tau} \frac{d\tau}{\sg + i\tau} + O\left(\sum_{m \geq 1} \frac{a(m)}{m^{\sg}}g_{x,L}(m)\right),
\end{equation*}
where $g_{x,L}(m) := \min\left\{1,\frac{1}{L|\log(x/m)|}\right\}$. Let $\kappa > 0$ be a parameter to be fixed momentarily, and let $\chi(t) := \max(1-|t|,0)$. Then whenever $|\log(x/m)|\leq \kappa$ we have $\chi\left(\frac{\log(x/m)}{2\kappa}\right) \geq \frac{1}{2}$. Hence,
\begin{align}
&\sum_{m \geq 1} \frac{a(m)}{m^{\sg}}g_{x,L}(m) \leq \sum_{m : |\log(x/m)| \leq \kappa} \frac{a(m)}{m^{\sg}} + \frac{1}{L\kappa}\sum_{m \geq 1} \frac{a(m)}{m^{\sg}} \nonumber\\
&\leq 2\left(\sum_{m \geq 1} \chi\left(\frac{\log(x/m)}{2\kappa}\right) \frac{a(m)}{m^{\sg}}+\frac{1}{L\kappa}F(\mbf{\rho},\sg)\right) \\
&= 2\left(\sum_{m \geq 1} \frac{a(m)}{m^{\sg}}\int_{-\infty}^{\infty}\hat{\chi}(u)e^{i\frac{\log(x/m)}{2\kappa}u} du + \frac{1}{L\kappa}F(\mbf{\rho},\sg)\right) \label{FUBINI}.
\end{align}
the last line following by the Fourier inversion theorem, since $\chi \in L^2(\mb{R})$.  Since
\begin{align*}
\hat{\chi}(u) &= \int_0^1 (1-t)\left(e^{iut} + e^{-iut}\right) = 2\int_0^1 (1-t)\cos(ut)dt = \frac{2}{u}\int_0^1 \sin(ut) dt \\
&= \frac{4}{u^2}\frac{1-\cos(u)}{2} = \left(\frac{\sin(u/2)}{u/2}\right)^2,
\end{align*}
and by Taylor's theorem, $\left(\frac{\sin(u/2)}{u/2}\right)^2 = \left(1-\frac{u^2}{24} + O\left(u^4\right)\right)^2 \ll \frac{1}{1+u^2}$ uniformly in $|u| < 1$; the bound $\hat{\chi}(u) \ll \frac{1}{1+u^2}$ is trivially satisfied for $|u| \geq 1$. Thus, $\hat{\chi} \in L^1(\mb{R})$, in fact, and we may apply Fubini's theorem in \eqref{FUBINI}, whence
\begin{align*}
&\sum_{m \geq 1} \frac{a(m)}{m^{\sg}}\min\left\{1,\frac{1}{L|\log(x/m)|}\right\} \leq 2\left(\int_{-\infty}^{\infty} x^{i\frac{u}{2\kappa}}\sum_{m \geq 1} \frac{a(m)}{m^{\sg + \frac{iu}{2\kappa}}} \hat{\chi}(u) + \frac{1}{L\kappa}F(\mbf{\rho},\sg)\right) \\
&\ll \int_{-\infty}^{\infty}\frac{du}{1+u^2} |F(\mbf{\rho};\sg + iu/2\kappa)| + \frac{F(\mbf{\rho};\sg)}{L\kappa} \asymp \int_0^{\infty} \frac{du}{1+u^2}|F\left(\mbf{\rho};\sg + iu/2\kappa\right)| + \frac{F(\mbf{\rho};\sg)}{L\kappa}.
\end{align*}
In preparation to apply Lemma \ref{LemSmallTrunc}, we break the integral into the regions $[0,T]$, $[T,L]$ and $[L,\infty]$, and make the change of variables $u \mapsto \frac{u}{2\kappa}$ in the first two. Thus,
\begin{equation*}
\int_{-\infty}^{\infty} \frac{du}{1+u^2} |F(\mbf{\rho};\sg + iu/2\kappa)| \ll 2\kappa(I_1 + I_2) + I_3,
\end{equation*}
where $I_1 := \int_0^{2T\kappa} |F(\mbf{\rho};\sg+iu')| du' \ll T\e|F(\mbf{\rho};\sg)|$, and from the second condition in \eqref{CASES},
\begin{align*}
I_2 &:= \int_{2T\kappa}^{2L\kappa} |F(\mbf{\rho};\sg+iu)| du \ll F(\mbf{\rho};\sg) \int_{2T\kappa}^{2L\kappa} \left(1+\left(\frac{\tau}{\sg-1}\right)^2\right)^{-\left(\sum_{0 \leq j \leq n} \frac{1}{\rho_j}\right)^{-1}} du \\
&\leq (\sg-1)F(\mbf{\rho};\sg) \int_{2\frac{T\kappa}{\sg-1}}^{\infty}(1+v^2)^{-\left(\sum_{0 \leq j \leq n} \frac{1}{\rho_j}\right)^{-1}} dv \\
&\ll (\sg-1)F(\mbf{\rho};\sg)\left(\sum_{0 \leq j \leq n} \frac{1}{\rho_j}\right)e^{-2\delta\kappa\left(\sum_{0 \leq j \leq n} \frac{1}{\rho_j}\right)^{-1}} \ll (\sg-1)\left(\sum_{0 \leq j \leq n} \frac{1}{\rho_j}\right)F\left(\mbf{\rho};\sg\right); \\
I_3 &:= \int_{L}^{\infty} |F(\mbf{\rho};\sg+iu/2\kappa)| \frac{du}{1+u^2} \ll F(\mbf{\rho};\sg)\int_L^{\infty} (1+u^2)^{-1}du\ll F(\mbf{\rho};\sg)L^{-1}.
\end{align*}
It follows that
\begin{equation*}
\sum_{m \geq 1} \frac{a(m)}{m^{\sg}}\min\left\{1,\frac{1}{L|\log(x/m)|}\right\} \ll F(\mbf{\rho};\sg)\left(\frac{1}{L\kappa} + \kappa\left(T\kappa + (\sg-1)\left(\sum_{0 \leq j \leq n} \frac{1}{\rho_j}\right)\right) + \frac{1}{L}\right).
\end{equation*}
Let $\eta$ be defined implicitly via $L^{-\eta} = \kappa$, with $0 < r < 1$.  As $T < \sg-1$ and $\kappa \ll \sum_{0 \leq j \leq n} \frac{1}{\rho_j}$, the third term and first are weakest. To make them equal, we choose $r = \frac{3}{7A}$, so that 
\begin{equation*}
(\sg-1)L^{-r} \ll L^{-r + \frac{1}{7A}} = L^{r - 1} = \log^{-B} x, 
\end{equation*}
where $B = A-\frac{3}{7}$. This last expression is thus bounded above by $O\left(\frac{F(\rho;\sg)}{\log^B x}\right)$. \\
Next, we investigate the difference between the truncated multiple integral above and the integral around the critical point $(\mbf{\rho};\sg)$. This difference in integrals is
\begin{align*}
&\left|\frac{1}{(2\pi)^{n+1}}\left(\int_{[-\pi,\pi]^{n+1}}d\mbf{t} \int_{|\tau| \leq L} F(\mbf{z};\sg + i\tau) x^{i\tau} \frac{d\tau}{\sg + i\tau} - \int_{\mc{B}} d\mbf{t}\int_{|\tau| \leq T} F(\mbf{z};\sg + i\tau) x^{i\tau} \frac{d\tau}{\sg + i\tau}\right)\right|\\
&\leq \frac{2}{(2\pi)^{n+1}}(J_1 + J_2),
\end{align*}
where we have defined 
\begin{align*}
J_1 &:=\int_{|\tau| \leq T} \frac{d\tau}{\sg + i\tau}\int_{\exists \ j : \theta_j < |t_j| \leq \pi}dt_0 \cdots dt_n  |F(\mbf{z};\sg + i\tau)|, \\
J_2 &:=  \int_{\forall \ j : |t_j| \leq \theta_j} dt_0 \cdots dt_n \int_{T<|\tau| \leq L} |F(\mbf{z};\sg + i\tau)| d\tau.
\end{align*}
Here, the factor of 2 in the second line follows because when both $|\tau| \geq T$ and $|t_j| > \theta_j$ for some $j$ then we get a smaller integral than either of $J_1$ and $J_2$ (as the decay in $|\tau|$ and $|t_j|$ is compounded in this case). We first evaluate $J_1$.  Fix $\tau$ for the moment with $|\tau| \leq T$. From the remarks following \eqref{OPT}, we get
\begin{align*}
\int_{|t_{j_0}| > \theta_{j_0}} dt_{j_0} e^{-\frac{1}{8}G(\mbf{t},\tau)} &\leq \int_{|t_{j_0}| > \theta_{j_0}} e^{-\frac{1}{2000}k_{j_0}t_{j_0}^2} dt_{j_0} \ll k_{j_0}^{-\frac{1}{2}}\int_{|u| > \frac{1}{2000}\theta_{j_0}k_{j_0}^{\frac{1}{2}}} e^{-u^2} du \\
&\ll \frac{1}{k_{j_0}^{\frac{1}{2}}} \left(k_{j_0}^{\frac{1}{2}}\theta_{j_0}e^{-\frac{1}{2000}k_{j_0}\theta_{j_0}^2}\right)
\end{align*}
(see \eqref{OPT}, where it is shown that this bound is optimal for the tail integral above); when $j$ is not such that $|t_j| > \theta_j$, we have
\begin{align*}
\int_{-\pi}^{\pi} dt_j e^{-\frac{1}{8}G(\mbf{t},\tau)} &\leq \int_{-\infty}^{\infty} du_j e^{-\frac{1}{2000}k_ju_j^2}  \ll k_j^{-\frac{1}{2}}\int_{-\infty}^{\infty} e^{-u^2} du \ll \left(\frac{2\pi}{k_j}\right)^{\frac{1}{2}}.
\end{align*}
It follows from Lemma \ref{LemInt} that
\begin{equation*}
\int_{[-\pi,\pi]^{n+1}} d\mbf{t} e^{-\frac{1}{8}G(\mbf{t},\tau)} \ll \left(\prod_{0 \leq j \leq n} k_j^{-\frac{1}{2}}\right)\left(\sum_{0 \leq j \leq n} k_j^{\frac{1}{2}}\theta_je^{-k_j\theta_j^2} \right),
\end{equation*}
and, as such, since $\int_{|\tau| \leq T} |F(\mbf{\rho};\sg+i\tau)| \leq TF(\mbf{\rho};\sg)$,
\begin{equation*}
J_1 \ll \int_{|\tau| \leq T} d\tau \int_{[-\pi,\pi]^{n+1}}d\mbf{t} |F(\mbf{z},\sg + i\tau)| \ll \delta F(\mbf{\rho},\sg)(\sg-1)\left(\prod_{0 \leq j \leq n} \frac{2\pi}{k_j}\right)^{\frac{1}{2}} \left(\sum_{0 \leq l \leq n} k_l^{\frac{1}{2}}\theta_le^{-k_l\theta_l^2} \right).
\end{equation*}
Now, consider $J_2$. We split this integral as 
\begin{align*}
J_2 &= \left(\int_{T < |\tau| \leq (\sg-1)} + \int_{(\sg-1) < |\tau| \leq 1} + \int_{1 < |\tau| \leq L}\right) d\tau \int_{|t_j| \leq \theta_j \ \forall j} d\mbf{t} |F(\mbf{z},\sg + i\tau)| =: \iota_1 + \iota_2 + \iota_3.	
\end{align*}
We first evaluate $\iota_1$. By Lemma \ref{LemSmallTrunc} in the range $\delta(\sg-1) < |\tau| \leq (\sg-1)$, we have
\begin{align*}
\iota_1 &\ll F\left(\mbf{\rho};\sg\right) \int_{\delta(\sg-1)}^{(\sg-1)} d\tau e^{-\frac{1}{250}\left(\sum_{0 \leq j \leq n} \rho_j \beta_j\left(e^{\frac{1}{\sg-1}}\right)\right)\frac{\tau^2}{(\sg-1)^2}} \int_{[-\pi,\pi]^{n+1}} d\mbf{t}e^{-\frac{1}{250}\sum_{0 \leq j \leq n} k_jt_j^2} \\
&\ll F\left(\mbf{\rho};\sg\right)\left(\prod_{0 \leq j \leq n}  \frac{2\pi}{k_j}\right)^{\frac{1}{2}}(\sg-1)\int_{\delta}^1 e^{-\frac{1}{250}\left(\sum_{0 \leq j \leq n} \rho_j\beta_j\left(e^{\frac{1}{\sg-1}}\right)\right)\tau^2} \\
\end{align*}
Proceeding as in Lemma \ref{LemInt}, set $D := \frac{1}{(\sg-1)^2}\sum_{0 \leq j \leq n} \rho_j\beta_j\left(e^{\frac{1}{\sg-1}}\right)$. Then, as $\delta = \frac{T}{\sg-1}$, we see that 
\begin{equation*}
\iota_1 \ll D^{-\frac{1}{2}}F\left(\mbf{\rho};\sg\right)\left(\prod_{0 \leq j \leq n}  \frac{2\pi}{k_j}\right)^{\frac{1}{2}}(\sg-1)\int_{\frac{1}{250}T D^{\frac{1}{2}}}^{\frac{1}{250}D^{\frac{1}{2}}} e^{-\tau^2}.
\end{equation*}
This last integral is bounded above, as in \eqref{FIRSTSMALL}, by $TD^{\frac{1}{2}}e^{-\frac{1}{2}T^2D}$, which implies that 
\begin{equation*}
\iota_1 \ll \delta(\sg-1)^2F\left(\mbf{\rho};\sg\right)\left(\prod_{0 \leq j \leq n}  \frac{2\pi}{k_j}\right)^{\frac{1}{2}}e^{-\frac{1}{2}T^2D}.
\end{equation*}
Now, as in Lemma \ref{LemInt}, $\delta \ll \log^{\e}\left(\frac{1}{\sg-1}\right) \left(\sum_{0 \leq j \leq n} \rho_j\beta_j\right)^{-\frac{1}{2}}$. As $\log^{\e}\left(\frac{1}{\sg-1}\right)(\sg-1) = o(1)$ and $e^{-\frac{1}{2}T^2D} \ll E_{j_0}\left(e^{\frac{1}{\sg-1}}\right)^{-\frac{1}{2}}$, as in Lemma \ref{LemInt}, we have $\iota_1 \ll I\left(\sum_{0\leq j \leq n} E_j(x)^{-\frac{1}{2}}\right)$, a fortiori. \\
We next evaluate $\iota_3$.  By Lemma \ref{GenIntEst} applied with $\gamma := \frac{1}{3072e}\left(\sum_{0 \leq j \leq n}\frac{(192e)^2}{\rho_j}\right)^{-1}$ and $u := \frac{1}{\sg-1}$, 
\begin{equation*}
\int_1^L e^{-\gamma \log\left(1+\frac{\tau^2}{(\sg-1)^2}\right)} d\tau \ll (\sg-1)^{\frac{1}{2}\gamma} \ll (\sg-1)^3\left(\prod_{0 \leq j \leq n} \frac{1}{2\pi k_j}\right)^{\frac{1}{2}},
\end{equation*}
this last estimate following from $(\sg-1) \ll k_j^{-\frac{1}{2}}$. Note that 
\begin{equation*}
\left(\sum_{0 \leq j \leq n} \rho_j\beta_j\right)^{-\frac{1}{2}} \gg \left(\sum_{0 \leq j \leq n} k_j\right)^{-\frac{1}{2}} \gg (\sg-1).
\end{equation*}
It follows that
\begin{align*}
\iota_3 &\ll \int_1^L d\tau \int_{[-\pi,\pi]^{n+1}} d\mbf{t} |F(\mbf{z};s)| \ll (\sg-1)^2F(\mbf{\rho};\sg)\left(\prod_{0 \leq j \leq n} \frac{2\pi}{k_j}\right)^{\frac{1}{2}}\left(\sum_{0 \leq j \leq n} \rho_j\beta_j\right)^{-\frac{1}{2}} \\
&\ll I(\sg-1).
\end{align*}
It remains to estimate $\iota_2$.  We observe, by Lemma \ref{EXTRADECAY} that $R(\mbf{t},\tau) \geq \log^2\left(\frac{1}{\sg-1}\right)$ once $|\tau| < 1$ since, for some $l$, $k_l \gg E_l(x)^2 \gg \log^2\left(\frac{1}{\sg-1}\right)$.  Thus, $e^{-\frac{1}{12}R(\mbf{t},\tau)} \ll (\sg-1)^3$. Now, the integral over $\mbf{t}$ is at most $2^{n+1}\prod_{0 \leq j \leq n} \theta_j \leq 2^{n+1}\prod_{0 \leq j \leq n} \frac{E_j\left(e^{\frac{1}{\sg-1}}\right)^{\e}}{k_j^{\frac{1}{2}}}$. Hence, we have
\begin{align*}
\iota_2 &\ll_n (\sg-1)^3F(\mbf{\rho};\sg)\left(\prod_{0 \leq j \leq n} k_j^{-\frac{1}{2}}E_j\left(e^{\frac{1}{\sg-1}}\right)^{\e}\right)\int_{\sg-1}^1  d\tau e^{-\gamma \log\left(1+\frac{\tau^2}{(\sg-1)^2}\right)} \\
&\ll (\sg-1)^3F(\mbf{\rho};\sg)\left(\prod_{0 \leq j \leq n} \frac{2\pi}{k_j}\right)^{\frac{1}{2}} \int_1^{\infty} e^{-\gamma \log\left(1+\tau^2\right)} d\tau \ll I(\sg-1),
\end{align*}
this last line following again because $(\sg-1) \ll \left(\sum_{0 \leq j \leq n} \rho_j\beta_j\right)^{-\frac{1}{2}}$.\\
Collecting all the integral estimates, we note that $J_1 \ll J_2$, so $J_1 + J_2 \ll I\sum_{0 \leq j \leq n} S_j$. The claim now follows from the trivial observation that 
\begin{equation*}
1+O\left(\sum_{0 \leq j \leq n} S_j\right) \leq \prod_{0 \leq j \leq n} \left(1+O(S_j)\right).
\end{equation*}
\end{proof}
\section{Completion of the Proof of Theorem \ref{THM1}}
\begin{proof}[Proof of Theorem \ref{THM1}]
Let $Y := e^{\frac{1}{\sg-1}}$. Recall that $\eta_j = k_j^{-1}\rho_jE_j\left(Y\right)-1$, and set 
\begin{equation*}
C(\mbf{\rho}) := (\sg-1)\log x = \sum_{0\leq j \leq n}\left(1+O\left(E_j(x)^{-1}\right)\right) \rho_j\alpha_j; 
\end{equation*}
note that by Lemma \ref{PARAMS}, $\eta_j \ll \frac{E_j(k_j/E_j(Y))}{E_j(Y)}$ and $\left\|\frac{1}{\mbf{\rho}}\right\|^{-1} \ll C(\mbf{\rho}) \ll \|\mbf{\rho}\|$.  By Stirling's approximation, we have $k_j! = k_j^{k_j+\frac{1}{2}}(2\pi)^{\frac{1}{2}} e^{-k_j}(1+O(k_j^{-1}))$, so that
\begin{align}
&(2\pi)^{-(n+1)}\prod_{0 \leq j \leq n}\rho_j^{-k_j}\left(\frac{2\pi}{k_j}\right)^{\frac{1}{2}} \nonumber\\
&= (2\pi)^{-\frac{n+1}{2}}\exp\left(\sum_{0 \leq j \leq n} \left(-k_j\log\left(\frac{\rho_jE_j\left(Y\right)}{k_j}\right) - \left(k_j+\frac{1}{2}\right)\log k_j + k_j\log E_j\left(Y\right) \right)\right) \nonumber\\
&= (2\pi)^{-\frac{n+1}{2}} \prod_{0 \leq j \leq n}\left(1+O\left(\frac{1}{E_j(x)^{2}}\right)\right) \exp\left(-k_j(\log(1+\eta_j)+1-\log E_j\left(Y\right)) - \log(k_j!)\right) \nonumber\\
&= \prod_{0 \leq j \leq n} \left(1+O\left(\frac{1}{E_j(x)^{2}}\right)\right)\frac{E_j\left(Y\right)^{k_j}}{k_j!}e^{-k_j(1+\log(1+\eta_j))} \label{PRODEXP}.
\end{align}
Since the sets $E_j$ form a partition,
\begin{equation*}
\sum_{0 \leq j \leq n} E_j\left(Y\right)  = \sum_{p \leq Y} \frac{1}{p} = \log\left(\frac{1}{\sg-1}\right) + M + O\left(\sg-1\right),
\end{equation*}
as in \eqref{MERTENS}. Let $\mc{M} := \prod_{0 \leq j \leq n} \rho_j^{-k_j}\left(\frac{2\pi}{k_j}\right)^{\frac{1}{2}}\left(1+O_{\e}\left(S_j\right)\right)$. Combining Lemmas \ref{LemInt} and \ref{LemAppTrunc} with \eqref{PRODEXP} and replacing $\e+\eta$ with a different parameter which we shall also denote by $\e$ (and which we allow to be arbitrarily small), we derive that when $S_j := E_j(Y)^{-\frac{1}{2}+\e}$, 
\begin{align*}
\pi(x;\mbf{E};\mbf{k}) &= x^{\sg}F\left(\mbf{\rho};\sg\right)\left(\frac{(\sg-1)}{(2\pi)^{n+2}}\mc{M}\left(\frac{2\pi}{\sum_{0 \leq j \leq n} \rho_j \beta_j}\right)^{\frac{1}{2}} +O\left(\frac{1}{\log^B x}\right)\right)\\
&= x\left(\prod_{0 \leq j \leq n}\left(1+O\left(E_j\left(Y\right)^{-\frac{1}{2}+\e}\right)\right) \frac{E_j\left(Y\right)^{k_j}}{k_j!}e^{-E_j\left(Y\right)}\right)\mc{F}(\mbf{k};\sg) + O\left(x\frac{F(\mbf{\rho},\sg)}{\log^{B-\frac{1}{2}} x}\right),
\end{align*}
where we have written
\begin{equation*}
\mc{F}(\mbf{k};\sg) := \left(\frac{e^{C(\mbf{\rho}) +M}}{\sqrt{2\pi}}\frac{F(\mbf{\rho},\sg)e^{-\sum_{0 \leq j \leq n} k_j(1+\log(1+\eta_j))}}{\left(\sum_{0 \leq j \leq n} \rho_j \beta_j\right)^{\frac{1}{2}}}\right).
\end{equation*}
We now estimate $\mc{F}(\mbf{k};\sg)$. Define $\mc{R}$ implicitly via $e^{\mc{R}} = F(\mbf{\rho};\sg)e^{-\sum_{0 \leq j \leq n} k_j(1+\log(1+\eta_j))}$. Thus, 
\begin{equation*}
\mc{F}(\mbf{k};\sg) := \left(2\pi \sum_{0 \leq j \leq n} \rho_j\beta_j\right)^{-\frac{1}{2}}\exp\left(\sum_{0 \leq j \leq n}\left(1+O\left(E_j(x)^{-1}\right)\right) \rho_j \alpha_j + \mc{R} + M\right).
\end{equation*}
Observe that, by the definition of $\eta_j$ and the identities
\begin{equation*}
\log\left(1+\frac{\rho_j}{p^{\sg}-1}\right) = -\log\left(1-\frac{\rho_j}{p^{\sg}-1+\rho_j}\right),
\end{equation*}
and $k_j = \rho_j \sum_{p \in E_j} \frac{1}{p^{\sg}-1+\rho_j}$, we have
\begin{align*}
&\mc{R} = \exp\left(-\sum_{0 \leq j \leq n} \left(\sum_{p \in E_j} \log\left(1-\frac{\rho_j}{p^{\sg}-1+\rho_j}\right) - k_j - k_j\log(1+\eta_j)\right)\right) \\
&= \exp\left(\sum_{0 \leq j \leq n} \left(\log\left(1-\frac{\rho_j}{p^{\sg}-1+\rho_j}\right)  - \rho_j\sum_{p \in E_j} \frac{1}{p^{\sg}-1+\rho_j} - k_j\log\left(1+\eta_j\right)\right) \right) \\
&= \exp\left(\sum_{0 \leq j \leq n} \sum_{l \geq 2} \frac{\rho_j^l}{l}\sum_{p \in E_j} \frac{1}{(p^{\sg}-1+\rho_j)^l} - \sum_{0 \leq j \leq n} k_j\log\left(1+\eta_j\right)\right).
\end{align*}
Now, by Lemma \ref{SIMPLE},
\begin{equation*} 
\rho_j^l\sum_{p \in E_j \atop p > \rho_j^{\frac{1}{\sg}}} \frac{1}{(p^{\sg}-1+\rho_j)^l} \ll \frac{1}{l}\frac{\rho_j}{\log \rho_j},
\end{equation*}
from which it follows that
\begin{align*} 
\sum_{l \geq 2} \frac{\rho_j^l}{l}\sum_{p \in E_j \atop p > \rho_j^{\frac{1}{\sg}}} \frac{1}{p^{\sg}-1+\rho_j}  \ll \frac{\rho_j}{\log \rho_j}\left(\sum_{l \geq 2} \frac{1}{l^2}\right) \ll \frac{\rho_j}{\log \rho_j}.
\end{align*}
On the other hand, when $p \leq \rho_j^{\frac{1}{\sg}} - 1$ then $(p^{\sg}-1+\rho_j)^l \leq 2^{l-1}\rho_j^{l-1}$, whence
\begin{equation*}
\sum_{l \geq 2} \frac{\rho_j^l}{l} \sum_{p \in E_j \atop p \leq \rho_j^{\frac{1}{\sg}}-1} \frac{1}{(p^{\sg}-1+\rho_j)^l} \geq \left(\sum_{l \geq 2} \frac{1}{l2^{l-1}}\right)\rho_j\sum_{p \in E_j \atop p \leq \rho_j^{\frac{1}{\sg}}-1} \frac{1}{p^{\sg}-1+\rho_j}.
\end{equation*}
Set $\phi:= \sum_{l \geq 2} \frac{1}{l2^{l-1}} > \frac{1}{4}$. As $\phi\rho_j\sum_{p \in E_j \atop p > \rho_j^{\frac{1}{\sg}}} \frac{1}{p^{\sg}-1+\rho_j} \ll \frac{\rho_j}{\log \rho_j}$ by Lemma \ref{SIMPLE}, it follows that
\begin{align*}
\mc{R} &\geq \exp\left(\phi\sum_{0 \leq j \leq n} \left(\rho_j\sum_{p \in E_j} \frac{1}{p^{\sg}-1+\rho_j} - k_j\log\left(1+\eta_j\right)\right) + O\left(\frac{\rho_j}{\log \rho_j}\right)\right) \\
&= \exp\left(\sum_{0 \leq j \leq n} k_j\left(\phi-\log\left(1+\eta_j\right)\right) + O\left(\frac{\rho_j}{\log \rho_j}\right)\right).
\end{align*}
This completes the proof of Theorem 1.
\end{proof}
\section{Appendix 1: $E_2$ in Theorem \ref{THMAPP}}
We can express the set $E_2$ mentioned in Theorem \ref{THMAPP} as follows.  For each $k > e^{e^{100}}$, let $r(k) := \llf \frac{\log_3 k}{\log 10}\rrf$, and let $a_0(k),\ldots, a_{r(k)}(k) \in \{0,\ldots,9\}$ be such that $a_0(k) \neq 0$ for each $k$. Write $A_k := \sum_{0 \leq j \leq r(k)} a_j(k)10^{k-j}$, and $S_k := [A_k,A_k + 10^{k-r(k)})$. Let $\mc{Q}$ denote the set of primes not congruent to 3 modulo 10. Then $E_2 := \mc{Q}\cap \left(\bigcup_{k \geq 1} S_k\right)$. Note that there are only 4 coprime residue classes modulo 10, and $\mc{Q}$ contains 3 of them. We thus have
\begin{align*}
\sum_{p \leq x \atop p \in E_2} \frac{1}{p} &\sim \sum_{e^{100} \leq k \leq \llf \frac{\log x}{\log 10}\rrf} \sum_{p \in S_k \cap \mc{Q}} \frac{1}{p} \sim \frac{3}{4}\sum_{e^{100} \leq k \leq \llf \frac{\log x}{\log 10}\rrf} \log\left(\frac{\log\left(A_k + 10^{k-r(k)}\right)}{\log A_k}\right) \\
&\sim \frac{3}{4}\sum_{e^{100} \leq k \leq \llf \frac{\log x}{\log 10}\rrf} \log\left(1+\frac{10^{k-r(k)}}{A_k \log A_k}\right) \asymp \sum_{e^{100} \leq k \leq \llf \frac{\log x}{\log 10}\rrf} \log\left(1+\frac{1}{a_0(k)10^{r(k)}\log A_k}\right) \\
&\asymp \sum_{e^{100} \leq k \leq \llf \frac{\log x}{\log 10}\rrf} \frac{1}{10^{r(k)}\log A_k} \asymp \frac{1}{\log 10}\sum_{e^{100} < k \leq \llf \frac{\log x}{\log 10}\rrf} \frac{1}{10^{r(k)}k} \\
&\sim \frac{1}{\log 10} \sum_{e^{100} < k \leq \llf \frac{\log x}{\log 10}\rrf} \frac{1}{k \log_2 k} \sim \int_{e^{100}}^{\frac{\log x}{\log 10}} \frac{dt}{t\log_2 t} \sim \int_{100}^{\log_2 x} \frac{du}{\log u} \sim \frac{\log_2 x}{\log_3 x}.
\end{align*}
Hence, $E_2(x) \asymp \frac{\log_2 x}{\log_3 x}$ for $x$ sufficiently large, as claimed.  
\section{Appendix 2: A Partition such that $\mathbf{H}_1(\sg)$ Fails}
In this section, we will show that there are infinitely many choices of $\sg$ such that no subset of a partition $\{E_1,\ldots,E_m\}$ simultaneously satisfies $E_j\left(e^{\frac{1}{\sg-1}}\right) \gg_m \log\left(\frac{1}{\sg-1}\right)$ and $\sum_{p \in E_j \atop p \leq e^{\frac{1}{\sg-1}}} \frac{\log p}{p^{\sg}} \gg_n \frac{1}{\sg-1}$.  This implies that hypothesis $\mathbf{H}_1(\sg)$ is non-trivial.\\
Let $\sg > 1$ be such that $\sg \ra 1^+$, and set $L := \llf \frac{\log\left(\frac{1}{\sg-1}\right)}{\log 2}\rrf$. For each $0 \leq l \leq L$, define $y_l := 2^{2^{l}}$, and for $0 \leq k \leq 2^{l-1}-1$, let $x_{l,k} := 2^{2^{l-1}+k}$. Thus, $y_{l-1} \leq x_{l,k} < y_l$ for each $0 \leq k \leq 2^{l-1}-1$.  Set $I_l := [y_{l-1},y_l)$ for each $1 \leq l \leq L$, and $J_{l,k} := [x_{l,k},x_{l,k+1})$. \\
Given $0 \leq j \leq n$, define $S_j := \{1 \leq l \leq L : E_j(y_l)-E_j(y_{l-1}) \gg_n 1\}$, and if $l \in S_j$ define $S_{j,l} := \{0 \leq k \leq 2^{l-1}-1 : E_j(x_{l,k+1})-E_j(x_{l,k}) \gg_n 2^{-(l-1)}\}$. Note that $E_j\left(e^{\frac{1}{\sg-1}}\right) \gg_n \log\left(\frac{1}{\sg-1}\right)$ if, and only if, $|S_j| \gg_n L$ by necessity (since $E_j(y_l)-E_j(y_{l-1}) \leq \log 2$ anyway). Moreover, since, for $l$ fixed, the intervals $J_{l,k}$ partition $I_l$, $|S_{j,l}| \gg_n 2^{l-1}$, whenever these sets are defined.  \\
When $l \leq L$, $p^{\sg} \leq e p$ whenever $p \in I_l$. Thus, in what follows, it will suffice to consider sums in terms of $\frac{\log p}{p}$. Note that whenever $k \in S_{j,l}$, 
\begin{equation*}
\sum_{p \in E_j \cap J_{l,k}} \frac{\log p}{p} \geq \left(2^{l-1}+k\right)\log 2 \left(E_j(x_{l,k+1})-E_j(x_{l,k})\right) \gg_n 1. 
\end{equation*}
Hence, 
\begin{equation} 
\sum_{p \in E_j \atop p \leq e^{\frac{1}{\sg-1}}} \frac{\log p}{p} \geq \sum_{l \in S_j} \sum_{k \in S_{j,l}} \sum_{p \in J_{l,k}} \frac{\log p}{p^{\sg}} \gg_n \sum_{l \in S_j} |S_{j,l}| \gg_n \sum_{l \in S_j} 2^{l-1} \label{LOGCONDITION}.
\end{equation}
With these ideas in mind, let $A,B \subset \mb{N}$, where $B := \bigcup_{l \geq 0} [2^{l+1}-l,2^{l+1})$ and $A := \mb{N} \bk B$.
Let $E_1 := \mc{P} \cap \left(\bigcup_{l \in A} I_l\right)$ and $E_2 := \mc{P} \cap \left(\bigcup_{l \in B} I_{l}\right)$. Thus, $\{E_1,E_2\}$ partition $\mc{P}$.  
Now, since $|B \cap [1,L]| \leq \sum_{l \leq \log L} l \ll \log^2 L$, it follows that $|S_1| = (1-o(1))L$, while $|S_2| = o(L)$. By the above remarks, this clearly shows that $E_2\left(e^{\frac{1}{\sg-1}}\right) = o\left(\log\left(\frac{1}{\sg-1}\right)\right)$, while $E_1\left(e^{\frac{1}{\sg-1}}\right) \ll \log\left(\frac{1}{\sg-1}\right)$, for any $\sg$. Now, suppose that $\sg$ is chosen such that $L := 2^M-1$, for some $M \in \mb{N}$, assumed large.  Then the maximal element of $S_1$ is $L-M+1$, while that of $S_2$ is $L$.  Hence, by \eqref{LOGCONDITION}, $\sum_{p \in E_2 \atop p \leq e^{\frac{1}{\sg-1}}} \frac{\log p}{p^{\sg}} \gg 2^L = \frac{1}{\sg-1}$.  On the other hand, setting $L' := 2^M-M+1$,
\begin{equation*}
\sum_{p \in E_1 \atop p \leq e^{\frac{1}{\sg-1}}} \frac{\log p}{p^{\sg}} \leq \sum_{p \leq y_{L'}} \frac{\log p}{p} \ll 2^{2^M-M} \ll L^{-1}2^L = o\left(\frac{1}{\sg-1}\right).
\end{equation*}
Since $M$ is an arbitrary, large integer, this provides a counterexample for infinitely many choices of $\sg$ (and hence $x$).
\subsection*{Acknowledgements}
The author would like to warmly thank his Ph.D thesis supervisor Dr. J. Friedlander for his generosity of time, patience and encouragement throughout the writing of this paper.
\bibliographystyle{acm}
\bibliography{bibManySets2}

\begin{thebibliography}{10}

\bibitem{D}
{\sc DELANGE, H.}
\newblock Sur des formules d'{A}tle {S}elberg.
\newblock {\em Acta Arithmetica 19\/} (1971), 105--146.

\bibitem{Ha}
{\sc HAL\'{A}SZ, G.}
\newblock On the distribution of additive and the mean values of multiplicative
  arithmetic functions.
\newblock {\em Stud. Sci. Math. Hung. 6\/} (1971), 211--233.

\bibitem{HR}
{\sc HARDY, G., and RAMANUJAN, S.}
\newblock The normal number of prime factors of a number $n$.
\newblock {\em Quaterly Journal of Mathematics 48\/} (1917), 76--92.

\bibitem{HaW}
{\sc HARDY, G., and WRIGHT, E.}
\newblock {\em An Introduction to the Theory of Numbers}.
\newblock Oxford, UK, 1956.

\bibitem{He}
{\sc HENSLEY, D.}
\newblock The distribution of round numbers.
\newblock {\em Proc. London Math. Soc. 54\/} (1987), 412--444.

\bibitem{HiT}
{\sc HILDEBRAND, A., and TENENBAUM, G.}
\newblock On the number of prime factors of an integer.
\newblock {\em Duke Math. J. 56 (3)\/} (1988), 471--501.

\bibitem{Ker}
{\sc KERNER, S.}
\newblock {\em R\'{e}partition d'entiers avec contraintes sur les diviseurs}.
\newblock PhD thesis, Universit\'{e} Henri Poincar\'{e}, Nancy 1, Nancy,
  France, 2002.

\bibitem{next}
{\sc MANGEREL, A.}
\newblock On the distribution of integers with restricted prime factors {II}.
\newblock in preparation.

\bibitem{N}
{\sc NORTON, K.}
\newblock On the number of restricted prime factors of an integer {I}.
\newblock {\em Illinois J. Math. 20\/} (1976), 681--705.

\bibitem{Pom}
{\sc POMERANCE, C.}
\newblock {\em On the distribution of round numbers}.
\newblock Springer Lecture Notes, Ootacamund, India, 1956.
\newblock in \emph{Number Theory}, edited by K. Alladi.

\bibitem{Sel}
{\sc SELBERG, A.}
\newblock Note on a paper by {L}.{G}. {S}athe.
\newblock {\em J. Indian Math. Soc. 18\/} (1954), 83--87.

\bibitem{Ten2}
{\sc TENENBAUM, G.}
\newblock {\em Introduction to Analytic and Probabilistic Number Theory}.
\newblock Cambridge University Press, Cambridge, UK, 1994.

\bibitem{Tu}
{\sc TUDESQ, C.}
\newblock Majoration de la loi locale de certaines fonctions additives.
\newblock {\em Archiv der Mathematik 67\/} (1996), 465--472.

\bibitem{Will}
{\sc WILLIAMS, K.}
\newblock Mertens' theorem for arithmetic progressions.
\newblock {\em J. Number Theory 6\/} (1974), 353--359.

\end{thebibliography}
\end{document}